\documentclass{amsart}
\usepackage{amsmath}
\usepackage{amscd}
\usepackage{amssymb}
\usepackage[mathscr]{eucal}
\usepackage{multirow}
\usepackage{bm}
\input xy
\xyoption{all}
\makeatletter
\@namedef{subjclassname@2020}{\textup{2020} Mathematics Subject Classification}
\makeatother

\newcommand{\bZ}{{\mathbb Z}}

\newcommand{\s}{\mathop{\mathrm{S}}\nolimits}
\newtheorem{thm}{Theorem}[section]
\newtheorem{prop}[thm]{Proposition}
\newtheorem{lemma}[thm]{Lemma}
\newtheorem{cor}[thm]{Corollary}

\newtheorem{thm'}{Theorem}[subsection]
\newtheorem{prop'}[thm']{Proposition}
\newtheorem{lemma'}[thm']{Lemma}
\newtheorem{cor'}[thm']{Corollary}
\newtheorem{exam'}[thm']{Example}
\allowdisplaybreaks

\numberwithin{equation}{section}
\author{Hideaki \=Oshima}
\address[H. \=Oshima]{Professor Emeritus, Ibaraki University, Mito, Ibaraki 310-8512, Japan}
\email{hideaki.ooshima.mito@vc.ibaraki.ac.jp}

\author{Katsumi \=Oshima}
\address[K. \=Oshima]{Ishikawa, Mito, Ibaraki 310-0905, Japan}
\email{k-oshima@mbr.nifty.com}

\begin{document}

\title{The generalized Hopf invariant and four-fold Toda brackets}
\maketitle
\begin{abstract}
We give two formulas for the generalized Hopf invariant and $4$-fold Toda brackets which are useful in computations of homotopy groups of spheres. 
\end{abstract}
\section{Introduction}
In \cite[Corollary 1.2]{OO4}, we have given a formula for 
the generalized Hopf invariant (alias James-Hopf invariant) $H$ 
and $n$-fold Toda brackets $\{\alpha_n,\dots,\alpha_1\}^{(\ddot{s}_t)}_{(m_n,\dots,m_1)}$ for all $n\ge 4$. 
In this paper we restrict our attention to some other $4$-fold Toda brackets and 
we modify and improve partially works of \^{O}guchi \cite{Og} and Mimura \cite{M1,M2}. 

We will work in $\mathrm{TOP}^*$, the category of pointed spaces and pointed maps, 
and we will use freely notations in \cite{OO1} except the order in numbering of spaces and maps. 
In particular $E^n$ denotes the $n$-fold reduced suspension, that is $E^nX=X\wedge\s^n$. 
We will also use symbols of elements of homotopy groups of spheres in \cite{T}. 

Our main results are 1.1--1.4 below. 

\begin{thm}
If a quadruple of homotopy classes $\alpha_k\in[E^{m_k}X_k,X_{k+1}]\ (1\le k\le 4)$ satisfies
\begin{gather}
E(\alpha_4\circ E^{m_4}\alpha_3)=0,\ \alpha_3\circ E^{m_3}\alpha_2=0,\ \alpha_2\circ E^{m_2}\alpha_1=0,\\
\{E\alpha_4,\alpha_3,E^{m_3}\alpha_2\}_{1+m_4}=\{0\},\quad\{\alpha_3,\alpha_2,E^{m_2}\alpha_1\}_{m_3}\ni 0,\\
\text{$X_k=\s^{n_k}$ for $k=3,5$ with $n_3\ge 1$ and $n_5\ge 1$},
\end{gather} 
then the $4$-fold Toda bracket $\{E\alpha_4,\alpha_3,\alpha_2,\alpha_1\}^{(2)}_{(1+m_4,m_3,m_2,m_1)}$ (see \S2.1 for definition) is not empty and 
\begin{align*}
H\{E\alpha_4,\alpha_3,\alpha_2,&\alpha_1\}_{(1+m_4,m_3,m_2,m_1)}^{(2)}\\
&\subset \bigcup_{\alpha\in\widetilde{\Delta}_{n_5}^{-1}([\alpha_4\circ E^{m_4}\alpha_3])}(-1)^{m_4} \{\alpha,EE^{m_3}\alpha_2,EE^{m_3}E^{m_2}\alpha_1\}_{1+m_4}.
\end{align*}
Here $\widetilde{\Delta}_{n_5}$ is a homomorphism from a subgroup of $\pi_{n_3+m_3+m_4+2}(\s^{2n_5+1})$ to a quotient group of $\pi_{n_3+m_3+m_4}(\s^{n_5})$ (see \S2.4) which relates with $\Delta$ of the $EH\Delta$-sequence of Toda, and $[\alpha_4\circ E^{m_4}\alpha_3]$ is the element 
represented by $\alpha_4\circ E^{m_4}\alpha_3$. 
\end{thm}

\begin{cor}
$\displaystyle{H\{E\nu_4,2\nu_7,\eta_{10}\circ\sigma_{11},\sigma_{18}\}_{(1,0,0,0)}^{(2)}=
\bar{\sigma}_9}$.
\end{cor}

\begin{thm}
Suppose that a sequence $\vec{\bm \alpha}=(\alpha_4,\alpha_3,\alpha_2,\alpha_1)$ of homotopy classes $\alpha_k\in[X_k,X_{k+1}]$ is quasi-admissible (see \S3.1 for definition), $X_k=\s^{n_k}\ (k=1,2,5)$ with $n_5\ge 1$, $X_\ell\ (\ell=3,4)$ is a CW complex with a vertex as the base point, and 
$$
1+n_1\le 2\cdot\mathrm{Min}\{\mathrm{conn}(X_4), \mathrm{conn}(X_4\cup_{f_3}CX_3)\},
$$
where $\mathrm{conn}(Y)$ denotes the connectivity of the space $Y$. 
Then $\{\vec{\bm \alpha}\}^{(2)'}$ (see \S3.1 for definition) is not empty and
$$
H\{\vec{\bm \alpha}\}^{(2)'}\subset \Delta^{-1}(\{\alpha_4,\alpha_3,\alpha_2\})\circ E^3\alpha_1,
$$
where $\Delta^{-1}(B)=\phi_{\s^{2n_5}}h_{n_5*}\partial^{-1}(B)$ for $B\subset\pi_{n_2+1}(\s^{n_5})$. 
Here see \S2.3 and \S2.4 for notations:
$$
\pi_{n_2+3}(\s^{2n_5+1})\overset{\phi_{\s^{2n_5}}}{\longleftarrow}\pi_{n_2+2}(\s^{2n_5}_\infty)\overset{h_{n_5*}}{\longleftarrow}\pi_{n_2+2}(\s^{n_5}_\infty,\s^{n_5})\overset{\partial}{\longrightarrow}\pi_{n_2+1}(\s^{n_5}).
$$
\end{thm}

\begin{cor}
\begin{enumerate}
\item $(\sigma_{12},\nu_{19},2\nu_{22},\nu_{25})$ is not admissible (see \S2.1 for definition)  but quasi-admissible. 
\item $H\{\sigma_{12},\nu_{19},2\nu_{22},\nu_{25}\}^{(2)'}=\nu_{25}^2$. 
\item If $\lambda_0\in\{\sigma_{12},\nu_{19},2\nu_{22},\nu_{25}\}^{(2)'}$, then $2(\lambda_0\circ\nu_{31})\equiv \sigma_{13}\circ\kappa_{20}\ \mathrm{mod}\ \sigma_{13}^3$ 
and the order of $\lambda_0\circ\nu_{31}$ is four. 
\end{enumerate}
\end{cor}

A similar formula to Theorem 1.1 was given by \^{O}guchi \cite[(7.18)]{Og} for his compositions $\{\alpha_4,E^n\alpha_3,E^n\alpha_2,E^n\alpha_1\}_n$. 
While the paper \cite{Og} is difficult to read because of a few of errors and gaps 
(see for example the last paragraph of Section 4 of \cite{OO1}), 
the definition of his composition is good and complicated (cf., \cite{OO1}). 
His compositions are generally subsets of $4$-fold Toda brackets 
$\{\vec{\bm \alpha}\}^{(2)}_{\vec{\bm m}}$ 
(see \S2.1 for definition) which are easily defined and may be alternatives to his compositions. 
Mimura, Mori and Oda \cite[p.20]{MMO} had the equality $H\{\nu_5,E(2\nu_7),E(\eta_{10}\circ\sigma_{11}),E(\sigma_{18})\}_1=\bar{\sigma}_9$ from  \cite[(7.18)]{Og} and used it  to compute $\pi_{28}(\s^5)$. 
We prefer the equality in Corollary 1.2 to their equality because of the reason stated above. 

Mimura \cite[p.175]{M1} defined 4-fold brackets and named them the higher compositions. 
His definition is incomplete (see for example lines 3--5 on p.57 of \cite{OO1}) and so some of his assertions are not effective. 
By the aid of his $4$-fold brackets he made assertions similar to our Corollary 1.4(2),(3) 
in \cite[p.308]{M2} and used them to compute $\pi_{n+21}(\s^n)$. 
By using Corollary 1.4 instead of his discussion \cite[p.308]{M2} 
we can follow his observation about the quadruple $(\sigma_{12},\nu_{19},2\nu_{22},\nu_{25})$ 
which is needed for Theorem~A of \cite{M2}. 
We should remark that there is a non-serious gap in the proof of the equality 
$\eta_6\circ\bar{\kappa}_7=E\alpha$ on \cite[p.306]{M2}. 
Mimura proved it from \cite[Lemma 6.1]{M1} where he asserted $\eta_7\circ\bar{\kappa}_8\in\{\nu^2_7,2\iota_{13},\kappa_{13}\}$ without proof. 
Nevertheless we can fill the gap with a correct proof which  
is typical and may be known for experts (cf.,\cite{Os}). 

Section 2 consists of four subsections \S2.1 -- \S2.4. 
We recall the definition of the $4$-fold Toda bracket $\{\alpha_4,\alpha_3,\alpha_2,\alpha_1\}_{\vec{\bm m}}^{(2)}$ in \S2.1 and the adjoint bijection $[X,\Omega^n Y]\to[E^nX,Y]$ in \S2.2. 
In \S2.3, we recall the definition of the generalized Hopf invariant and give a sufficient condition on a quadruple of maps to be admissible. 
In \S2.4, we prove Theorem 1.1 and Corollary 1.2. 
Section~3 consists of two subsections \S3.1 and \S3.2. 
In \S3.1 we define $\{\alpha_4,\alpha_3,\alpha_2,\alpha_1\}^{(2)'}$ and prove Theorem~1.3. 
In \S3.2 we prove Corollary 1.4. 
In Appendix~A we give an easy proof of \cite[Theorem 4.9]{OO1} which says that $\{\vec{\bm \alpha}\}^{(2)}_{\vec{\bm m}}$ is not empty if $\vec{\bm \alpha}$ is admissible. 


\section{$4$-fold Toda brackets $\{\vec{\bm \alpha}\}^{(2)}_{\vec{\bm m}}$ and the generalized Hopf invariant}

\numberwithin{equation}{subsection}
First we recall the notation of the classical Toda bracket. 
Given a triple $(\alpha_3,\alpha_2,\alpha_1)$ of homotopy classes $\alpha_k\in[E^{m_k}X_k,X_{k+1}]$, 
where $m_k$ are non negative integers, we denote by both $\{\alpha_3,\alpha_2,\alpha_1\}_{(m_3,m_2,m_1)}$ and $\{\alpha_3,\alpha_2,E^{m_2}\alpha_1\}_{m_3}$ the classical Toda bracket of which the original notation in \cite{T} is $\{\alpha_3,E^{m_3}\alpha_2,E^{m_3}E^{m_2}\alpha_1\}_{m_3} 
\,\big(\subset[EE^{m_3}E^{m_2}E^{m_1}X_1,X_4]\big)$. 

\subsection{Definition of $\{\alpha_4,\alpha_3,\alpha_2,\alpha_1\}^{(2)}_{(m_4,m_3,m_2,m_1)}$}
 
Suppose that a sequence of non negative integers $\vec{\bm m}=(m_4, m_3,m_2,m_1)$, five pointed spaces $X_k\ (5\ge k\ge 1)$, and a quadruple $\vec{\bm \alpha}=(\alpha_4,\alpha_3,\alpha_2,\alpha_1)$ of homotopy classes $\alpha_k\in[E^{m_k}X_k,X_{k+1}]$ 
are given. 
We use the notations: $|\vec{\bm m}|=\sum_k m_k$ and $E^{m_{[k,\ell]}}=E^{m_k}E^{m_{k-1}}\cdots E^{m_\ell}$ for $4\ge k>\ell\ge 1$. 
We call $\vec{\bm \alpha}$ (or precisely $\vec{\bm \alpha}_{\vec{\bm m}}$) a {\it null quadruple} if $\alpha_{k+1}\circ E^{m_{k+1}}\alpha_k=0\ (k=3,2,1)$, and 
{\it admissible} if there exist a representative $\vec{\bm f}=(f_4,f_3,f_2,f_1)$ (or precisely $\vec{\bm f}_{\vec{\bm m}}$) of $\vec{\bm \alpha}$, that is, $f_k$ represents $\alpha_k$, and null homotopies $A_k:f_{k+1}\circ E^{n_{k+1}}f_k\simeq *\ (k=1,2,3)$ such that 
$$
[f_{k+2},A_{k+1},E^{m_{k+2}}f_{k+1}]\circ (E^{m_{k+2}}f_{k+1},\widetilde{E}^{m_{k+2}}A_k, E^{m_{[k+2,k+1]}}f_k)\simeq *\ (k=1,2).
$$ 
In the last case, we call $\vec{\bm A}=(A_3,A_2,A_1)$ 
(or precisely $\vec{\bm A}_{\vec{\bm m}}$) an {\it admissible triple} for 
$\vec{\bm f}$, and we define
\begin{align*}
\{\vec{\bm f}\,;\vec{\bm A}\}^{(2)}_{\vec{\bm m}}&=\{f_4,[f_3,A_2,E^{m_3}f_2],(E^{m_3}f_2,\widetilde{E}^{m_3}A_1,E^{m_{[3,2]}}f_1)\}_{m_4} \\
&\hspace{1cm}\cap \{[f_4,A_3,E^{m_4}f_3]\circ(\psi^{m_4}_{f_3})^{-1},(f_3,A_2,E^{m_3}f_2), -EE^{m_{[3,2]}}f_1\}_{m_4}
\end{align*}
which is a non-empty subset of the abelian group $[E^2E^{m_{[4,1]}}X_1,X_5]$ by Proposition~A.2 below. 
Frequently we abbreviate $\{\vec{\bm f}\,;\vec{\bm A}\}^{(2)}_{\vec{\bm m}}$ to $\{\vec{\bm A}\}^{(2)}_{\vec{\bm m}}$. 
For any $\delta\in \{\vec{\bm A}\}^{(2)}_{\vec{\bm m}}$ we have  
\begin{align*}
&\{\vec{\bm A}\}^{(2)}_{\vec{\bm m}}\\
&=\big([E^{m_4+1}(E^{m_3}X_3\cup_{E^{m_3}f_2} CE^{m_{[3,2]}}X_2),X_5]\circ E^{m_4+1}(E^{m_3}f_2,\widetilde{E}^{m_3}A_1,E^{m_{[3,2]}}f_1)\\
&\hspace{2cm}+f_4\circ E^{m_4}[EEE^{m_{[3,1]}}X_1,X_4]+\delta \big)\\
&\qquad \cap\big([E^{m_4+1}EE^{m_{[3,2]}}X_2,X_5]\circ E^{m_4+1}(-EE^{m_{[3,2]}}f_1)\\
&\hspace{2cm}+[f_4,A_3,E^{m_4}f_3]\circ(\psi^{m_4}_{f_3})^{-1}\circ E^{m_4}[EEE^{m_{[3,1]}}X_1,X_4\cup_{f_3} CE^{m_3}X_3] +\delta\big)\\
&\supset [E^{m_4+1}EE^{m_{[3,2]}}X_2,X_5]\circ E^{m_4+1}(-EE^{m_{[3,2]}}f_1)
+f_4\circ E^{m_4}[EEE^{m_{[3,1]}}X_1,X_4] +\delta.
\end{align*}
Hence, as seen on \cite[p.35]{OO1}, 
\begin{equation}
\mathrm{Indet}\{\vec{\bm A}\}_{\vec{\bm m}}^{(2)}\supset[E^2E^{m_{[4,2]}}X_2,X_5]\circ E^2E^{m_{[4,2]}}f_1+f_4\circ E^{m_4}[E^2E^{m_{[3,1]}}X_1,X_4].
\end{equation}
We define
$$
\{\vec{\bm f}\}^{(2)}_{\vec{\bm m}}=\bigcup\{\vec{\bm A}\}^{(2)}_{\vec{\bm m}}
$$
where the union $\bigcup$ is taken over all admissible triples $\vec{\bm A}$ for $\vec{\bm f}$. 
If $\vec{\bm \alpha}$ is admissible, then for any representative $\vec{\bm f}$ of $\vec{\bm \alpha}$ there exists an admissible triple for $\vec{\bm f}$ by \cite[Proposition 2.11]{OO1} and $\{\vec{\bm f}\}^{(2)}_{\vec{\bm m}}$ depends only on $\vec{\bm \alpha}$ by \cite[(2.7), Lemma 2.9, Lemma 2.10]{OO1}. 
Therefore we can denote $\{\vec{\bm f}\}^{(2)}_{\vec{\bm m}}$ by $\{\vec{\bm \alpha}\}_{\vec{\bm m}}^{(2)}$. 
If $\vec{\bm \alpha}$ is not admissible, $\{\vec{\bm \alpha}\}_{\vec{\bm m}}^{(2)}$ denotes the empty set. 
When $m_k=0$ for all $k$, we abbreviate 
$\{\vec{\bm f}\,;\vec{\bm A}\}_{\vec{\bm m}}^{(2)}$, 
$\{\vec{\bm A}\}^{(2)}_{\vec{\bm m}}$, 
$\{\vec{\bm f}\}^{(2)}_{\vec{\bm m}}$ and $\{\vec{\bm \alpha}\}_{\vec{\bm m}}^{(2)}$ to $\{\vec{\bm f}\,;\vec{\bm A}\}^{(2)}$, 
$\{\vec{\bm A}\}^{(2)}$, 
$\{\vec{\bm f}\}^{(2)}$ and $\{\vec{\bm \alpha}\}^{(2)}$ respectively. 

\begin{lemma'}
Let $\vec{\bm f}=(f_4,\dots,f_1)$ be an admissible sequence of maps $f_k:E^{m_k}X_k\to X_{k+1}\ (1\le k\le 4)$ and $(A_3,A_2,A_1)$ an admissible triple for $\vec{\bm f}$. 
Then, given a map $h:X_5\to X_5'$, we have 
$$
h_*\{\vec{\bm f}\}^{(2)}_{\vec{\bm m}}\subset \{h\circ f_4,f_3,f_2,f_1\}^{(2)}_{\vec{\bm m}}.
$$ 
If moreover $h$ is a homotopy equivalence, then 
$$
h_*\{\vec{\bm f}\}^{(2)}_{\vec{\bm m}}= \{h\circ f_4,f_3,f_2,f_1\}^{(2)}_{\vec{\bm m}}.
$$ 
\end{lemma'}
\begin{proof}
Suppose that $\vec{\bm A}=(A_3,A_2,A_1)$ is an admissible triple for $\vec{\bm f}$. 
Then, as is easily seen, $(h\circ A_3,A_2,A_1)$ is an admissible triple for $(h\circ f_4,f_3,f_2,f_1)$ 
and $(\widetilde{E}A_3,A_2,A_1)$ is an admissible triple for $(Ef_4,f_3,f_2,f_1)$. 
We have 
\begin{align*}
h_*&\{\vec{\bm A}\}^{(2)}_{\vec{\bm m}}\\
&=h_*\big(\{f_4,[f_3,A_2,E^{m_3}f_2],(E^{m_3}f_2,\widetilde{E}^{m_3}A_1,E^{m_{[3,2]}}f_1)\}_{m_4}\\
&\hspace{2cm}\cap \{[f_4,A_3,E^{m_4}f_3]\circ(\psi^{m_4}_{f_3})^{-1},(f_3,A_2,E^{m_3}f_2),-EE^{m_{[3,2]}}f_1\}_{m_4}\big)\\
&\subset h_*\{f_4,[f_3,A_2,E^{m_3}f_2],(E^{m_3}f_2,\widetilde{E}^{m_3}A_1,E^{m_{[3,2]}}f_1)\}_{m_4}\\
&\hspace{2cm}\cap h_*\{[f_4,A_3,E^{m_4}f_3]\circ(\psi^{m_4}_{f_3})^{-1},(f_3,A_2,E^{m_3}f_2),-EE^{m_{[3,2]}}f_1\}_{m_4}\\
&\subset \{h\circ f_4,[f_3,A_2,E^{m_3}f_2],(E^{m_3}f_2,\widetilde{E}^{m_3}A_1,E^{m_{[3,2]}}f_1)\}_{m_4}\\
&\hspace{2cm}\cap \{h\circ[f_4,A_3,E^{m_4}f_3]\circ(\psi^{m_4}_{f_3})^{-1},(f_3,A_2,E^{m_3}f_2),-EE^{m_{[3,2]}}f_1\}_{m_4}\\
&= \{h\circ f_4,[f_3,A_2,E^{m_3}f_2],(E^{m_3}f_2,\widetilde{E}^{m_3}A_1,E^{m_{[3,2]}}f_1)\}_{m_4}\\
&\hspace{2cm}\cap \{[h\circ f_4,h\circ A_3,E^{m_4}f_3]\circ(\psi^{m_4}_{f_3})^{-1},(f_3,A_2,E^{m_3}f_2),-EE^{m_{[3,2]}}f_1\}_{m_4}\\
&=\{h\circ A_3,A_2,A_1\}^{(2)}_{\vec{\bm m}}.
\end{align*}
Hence the containment is proved. 
Suppose that $h$ is a homotopy equivalence and let $h^{-1}$ be a homotopy inverse of $h$. 
We have 
$$
\{\vec{\bm f}\}^{(2)}_{\vec{\bm m}}=h^{-1}_* h_*\{\vec{\bm f}\}^{(2)}_{\vec{\bm m}}
\subset h^{-1}_*\{h\circ f_4,f_3,f_2,f_1\}^{(2)}_{\vec{\bm m}}
\subset\{h^{-1}\circ h\circ f_4,f_3,f_2,f_1\}^{(2)}_{\vec{\bm m}}=\{\vec{\bm f}\}^{(2)}_{\vec{\bm m}}.
$$
Hence $\{\vec{\bm f}\}^{(2)}_{\vec{\bm m}}
= h^{-1}_*\{h\circ f_4,f_3,f_2,f_1\}^{(2)}_{\vec{\bm m}}$ and so, by applying $h_*$ from left, we have 
$h_* \{\vec{\bm f}\}^{(2)}_{\vec{\bm m}}=\{h\circ f_4,f_3,f_2,f_1\}^{(2)}_{\vec{\bm m}}$. 
This completes the proof of Lemma 2.1.1. 
\end{proof}

If $\vec{\bm \alpha}$ is admissible, then two Toda brackets $\{\alpha_4,\alpha_3,E^{m_3}\alpha_2\}_{m_4}$ and $\{\alpha_3,\alpha_2,E^{m_2}\alpha_1\}_{m_3}$ contain $0$. 
A sufficient condition that a null quadruple is admissible was given by \^{O}guchi \cite[(6.3)]{Og} (cf. \cite[Proposition 4.1]{OO1}). 
A particular case remarked 
by Mimura \cite[p.174]{M1} (cf. \cite[Proposition~4.4]{OO1}) is that if (i) or (ii) below holds, then $\vec{\bm \alpha}$ is admissible. 
\begin{align}
\{\alpha_4,\alpha_3,E^{m_3}\alpha_2\}_{m_4}=\{0\}\ \ \text{and}\ \ 
\{\alpha_3,\alpha_2,E^{m_2}\alpha_1\}_{m_3}\ni 0. \tag{i}\\
\{\alpha_4,\alpha_3,E^{m_3}\alpha_2\}_{m_4}\ni 0\ \ \text{and}\ \ 
\{\alpha_3,\alpha_2,E^{m_2}\alpha_1\}_{m_3}=\{0\}.\tag{ii}
\end{align}

It follows from \cite[Proposition 5.2]{OO4} that if $X_k$ is well pointed for all $k$ then $\{\vec{\bm f}\}^{(2)}_{\vec{\bm m}}\subset (-1)^{m_4}\{\vec{\bm f}\}^{(\ddot{s}_t)}_{\vec{\bm m}}$ (see \cite{OO3} or \cite{OO4} for the definition of $\{\vec{\bm f}\}^{(\ddot{s}_t)}_{\vec{\bm m}}$). 

\subsection{Adjoint $\Theta^n_{(X,Y)}:[X,\Omega^n Y]\cong[E^nX,Y]$}

We denote by $Y^X$ the space of maps from $X$ to $Y$ with the compact open topology. 
Set $\Omega^n Y=Y^{\s^n}$, where $\s^n=\{(x_1,\dots,x_{n+1})\in\mathbb{R}^{n+1}\,|\,\sum x^2_i=1\}$ with $*=(1,0,\dots,0)$ the base point. 
We defined the identifications 
$$
\s^n=\underbrace{\s^1\wedge\cdots\wedge\s^1}_n,\quad \s^{n_1}\wedge\s^{n_2}\wedge\cdots\wedge\s^{n_k}=\s^{n_1+\dots+n_k}
$$
in \cite{OO2} as follows. 
Let $\psi_1:(I,\partial I)\to(\s^1,*)$ be the canonical relative homeomorphism defined by $\psi_1(t)=(\cos2\pi t,-\sin2\pi t)=:\bar{t}$. 
For $n\ge 2$, let $\psi_n:(I^n,\partial I^n)\to(\s^n,*)$ be any but fixed relative homeomorphism for each $n$, for example, define it as in \cite[p.5]{T}. 
We identify $I^n/\partial I^n$ with $\s^n$ by the homeomorphism induced from $\psi_n$ for $n\ge 1$. 
We identify $\s^n$ with $\underbrace{\s^1\wedge\cdots\wedge\s^1}_n$ by the homeomorphism $h_n$ making the following diagram commutative.
$$
\begin{CD}
I^n@>\psi_1\times\cdots\times\psi_1>>\s^1\times\cdots\times\s^1\\
@V\psi_nVV @VV q V\\
\s^n@>h_n>\approx>\s^1\wedge\cdots\wedge\s^1
\end{CD}
$$ 
We identify $\s^{n_1}\wedge\cdots\wedge\s^{n_k}$ with $\s^{n_1+\dots+n_k}$ for $n_i\ge 1$ as follows. 
\begin{align*}
&(x_1\wedge\cdots\wedge x_{n_1})\wedge(x_{n_1+1}\wedge\cdots \wedge x_{n_1+n_2})\wedge\cdots\wedge (x_{n_1+\cdots+n_{k-1}+1}\wedge\cdots\wedge x_{n_1+\cdots+n_k}) \\
&=x_1\wedge x_2\wedge\cdots\wedge x_{n_1+\cdots+n_k}\hspace{3cm}  (x_i\in\s^1).
\end{align*}
Let $\theta_1:\s^1\to\s^1\vee\s^1$ be the standard comultiplication defined by 
$$
\theta_1(\bar{s})=\begin{cases}(\overline{2s},*) & 0\le s\le 1/2\\ (*,\overline{2s-1}) & 1/2\le s\le 1\end{cases}.
$$ 
For $n\ge 2$, let $\theta_n:\s^n\to\s^n\vee\s^n$ be the composite of 
$$
\s^n=\s^{n-1}\wedge\s^1\overset{1_{\s^{n-1}}\wedge\theta_1}{\longrightarrow}\s^{n-1}\wedge(\s^1\vee\s^1)=(\s^{n-1}\wedge\s^1)\vee(\s^{n-1}\wedge\s^1)=\s^n\vee\s^n.
$$
That is, $\theta_n(s_{n-1}\wedge\bar{s})=\begin{cases} (s_{n-1}\wedge\overline{2s},*) & 0\le s\le 1/2\\(*,s_{n-1}\wedge\overline{2s-1}) & 1/2\le s\le 1\end{cases}$, where 
we used the identification $\s^n=\s^{n-1}\wedge\s^1$. 
Then $\theta_n$ is a comultiplication on $\s^n$. 
For $n\ge 1$, let  $+:\Omega^n Y\times\Omega^n Y\to \Omega^n Y$ and $+:Y^{E^nX}\times Y^{E^nX}\to Y^{E^nX}$ be the 
standard operations defined by
\begin{gather*}
(\omega+\omega')(s_{n-1}\wedge\bar{s})=\begin{cases} \omega(s_{n-1}\wedge\overline{2s}) & 0\le s\le 1/2\\ \omega'(s_{n-1}\wedge\overline{2s-1}) & 1/2\le s\le 1\end{cases},\\ (a+b)(x\wedge s_{n-1}\wedge\bar{s})=\begin{cases} a(x\wedge s_{n-1}\wedge\overline{2s}) & 0\le s\le 1/2\\ b(x\wedge s_{n-1}\wedge\overline{2s-1}) & 1/2\le s\le 1\end{cases},
\end{gather*}
where $s_{n-1}\in\s^{n-1}$ and we take off $s_{n-1}$ if $n=1$. 
Let $\Theta_{(X,Y)}^n:(\Omega^n Y)^X\to Y^{E^nX}$ be defined by $\Theta_{(X,Y)}^n(f)(x\wedge s_n)=f(x)(s_n)\ (x\in X, s_n\in\s^n)$. 
Then it is a bijection and satisfies $\Theta_{(X,Y)}^n(f+g)=\Theta_{(X,Y)}^n(f)+\Theta_{(X,Y)}^n(g)$. 
Its inverse ${\Theta_{(X,Y)}^n}^{-1}:Y^{E^nX}\to(\Omega^n Y)^{X}$ is given by ${\Theta_{(X,Y)}^n}^{-1}(g)(x)(s_n)=g(x\wedge s_n)$. 
Moreover if $n\ge 1$ then $\Theta_{(X,Y)}^n$ induces an isomorphism of groups 
$$
\Theta_{(X,Y)}^n:[X, \Omega^n Y]\cong [E^nX, Y]
$$
of which the inverse $[E^nX,Y]\to[X,\Omega^n Y]$ is induced from ${\Theta_{(X,Y)}^n}^{-1}$. 
Let $ev_X:E^n\Omega^n X\to X$ be the evaluation, that is, $ev_X(\omega\wedge s_n)=\omega(s_n)$ for $\omega\in\Omega^n X$ and $s_n\in\s^n$. 

\begin{lemma'}
\begin{enumerate}
\item Given two maps $\Omega^n X_3\overset{f_2}{\longleftarrow}X_2\overset{f_1}{\longleftarrow}X_1\ (n\ge 1)$, we have 
\begin{equation}
\Theta_{(X_1,X_3)}^n(f_2\circ f_1)=\Theta_{(X_2,X_3)}^n(f_2)\circ E^nf_1,
\end{equation}
and, if moreover $f_2\circ f_1\simeq *$, then for any null homotopy 
$A:CX_1\to \Omega^n X_3$ of $f_2\circ f_1$ we have 
\begin{equation}
\Theta_{(X_2\cup_{f_1} CX_1,X_3)}^n([f_2,A,f_1])=[\Theta_{(X_2,X_3)}^n(f_2), ev_{X_3} \circ\widetilde{E}^nA,E^nf_1]\circ(\psi^n_{f_1})^{-1}.
\end{equation}
\item If the triple $(f_3,f_2,f_1)$ of $\Omega^n X_4\overset{f_3}{\longleftarrow}E^{m_3}X_3\ (n\ge 1)$ and $X_{k+1}\overset{f_k}{\longleftarrow}E^{m_k}X_k\ (k=1,2)$ is a null triple, that is, if $f_{k+1}\circ E^{m_{k+1}}f_k\simeq *\ (k=1,2)$,  then 
\begin{equation}
\begin{split}
\Theta^n_{(EE^{|\vec{\bm m}|}X_1,X_4)}&\{f_3,f_2,f_1\}_{\vec{\bm m}}\\
&= \{\Theta_{(E^{m_3}X_3,X_4)}^n(f_3), f_2, f_1\}_{(n+m_3,m_2,m_1)}
\circ(1_{E^{|\vec{\bm m}|}X_1}\wedge\tau(\s^1,\s^n))\\
&\subset[E^nEE^{|\vec{\bm m}|}X_1,X_4].
\end{split}
\end{equation}
\item Given a quadruple $\vec{\bm f}_{\vec{\bm m}}$ with $X_{k+1}\overset{f_k}{\longleftarrow} E^{m_k}X_k\ (1\le k\le 4)$ and $m_4\ge 1$, we set $f_4^*=\Theta^{-1}_{(E^{m_4-1}X_4,X_5)}(f_4):E^{m_4-1}X_4\to \Omega X_5$. 
We have
\begin{enumerate}
\item $\vec{\bm f}_{\vec{\bm m}}$ is a null quadruple if and only if $(f_4^*,f_3,f_2,f_1)_{(m_4-1,m_3,m_2,m_1)}$ is a null quadruple. 
\item $\vec{\bm f}_{\vec{\bm m}}$ is admissible if and only if $(f_4^*,f_3,f_2,f_1)_{(m_4-1,m_3,m_2,m_1)}$ is admissible. 
\item If $\vec{\bm f}_{\vec{\bm m}}$ is admissible, then 
\begin{equation}
\begin{split}
&\Theta_{(EE^{|\vec{\bm m}|}X_1,X_5)}\{f_4^*,f_3,f_2,f_1\}^{(2)}_{(m_4-1,m_3,m_2,m_1)}\\
&\hspace{2cm}=
\{f_4,f_3,f_2,f_1\}^{(2)}_{\vec{\bm m}}\circ(1_{E^{|\vec{\bm m}|}X_1}\wedge\tau(\s^1,\s^1)).
\end{split}
\end{equation}
\item If $X_5=E\s^n\ (n\ge 1)$ and $f_4=Ef'_4$ for some $f_4':E^{m_4-1}X_4\to \s^n$, then $\vec{\bm f}_{\vec{\bm m}}$ is admissible if and only if $(i\circ f_4',f_3,f_2,f_1)_{(m_4-1,m_3,m_2,m_1)}$ is admissible, where $i:\s^n\to \s^n_\infty$ is the inclusion which shall be defined in the next subsection. 
\end{enumerate}
Here ``null quadruple'' and ``admissible'' for quadruple of maps 
are defined in the obvious way i.e. 
by using $\simeq$ instead of $=$. 
\end{enumerate}
\end{lemma'}
\begin{proof}
(1) $\Theta_{(X_1,X_3)}^n(f_2\circ f_1)(x_1\wedge s_n)=(f_2\circ f_1)(x_1)(s_n)=f_2(f_1(x_1))(s_n)$ and $\Theta_{(X_2,X_3)}^n(f_2)\circ E^nf_1(x_1\wedge s_n)=f_2(f_1(x_1))(s_n)$. 
Hence $\Theta_{(X_1,X_3)}^n(f_2\circ f_1)=\Theta_{(X_2,X_3)}^n(f_2)\circ E^nf_1$. 
This proves (2.2.1). 
The equality (2.2.2) follows readily from definitions. 

(2) Let $A_2:f_3\circ E^{m_3}f_2\simeq *$ and $A_1:f_2\circ E^{m_2}f_1\simeq *$ be any null homotopies. 
We have 
\begin{align*}
&\Theta^n_{(EE^{|\vec{\bm m}|}X_1,X_4)}([f_3,A_2,E^{m_3}f_2]\circ(E^{m_3}f_2,\widetilde{E}^{m_3}A_1,E^{m_{[3,2]}}f_1))\\
&=\Theta^n_{(E^{m_3}X_3\cup_{E^{m_3}f_2}CE^{m_{[3,2]}}X_2,X_4)}([f_3,A_2,E^{m_3}f_2])\circ E^n(E^{m_3}f_2,\widetilde{E}^{m_3}A_1,E^{m_{[3,2]}}f_1)\quad(\text{by (2.2.1)})
\\
&=[\Theta^n_{(E^{m_3}X_3,X_4)}(f_3),ev_{X_4}\circ\widetilde{E}^nA_2,E^nE^{m_3}f_2]\circ(\psi^n_{E^{m_3}f_2})^{-1}\\&\hspace{2cm}\circ(\psi^n_{E^{m_3}f_2})\circ(E^nE^{m_3}f_2,\widetilde{E}^n\widetilde{E}^{m_3}A_1,E^nE^{m_{[3,2]}}f_1)
\circ(1_{E^{m_{[3,1]}}X_1}\wedge\tau(\s^1,\s^n))\\
&\hspace{3cm}(\text{by (2.2.2) and \cite[Lemma 2.4]{OO1}})\\
&=[\Theta^n_{(E^{m_3}X_3,X_4)}(f_3),ev_{X_4}\circ\widetilde{E}^nA_2,E^nE^{m_3}f_2]\circ (E^nE^{m_3}f_2,\widetilde{E}^n\widetilde{E}^{m_3}A_1,E^nE^{m_{[3,2]}}f_1)\\
&\hspace{3cm}\circ (1_{E^{m_{[3,1]}}X_1}\wedge\tau(\s^1,\s^n))\\
&\in\{\Theta^n_{(E^{m_3}X_3,X_4)}(f_3),f_2,f_1\}_{(n+m_3,m_2,m_1)}\circ (1_{E^{m_{[3,1]}}X_1}\wedge\tau(\s^1,\s^n))
\end{align*}
and so 
$$
\Theta^n_{(EE^{|\vec{\bm m}|}X_1,X_4)}\{f_3,f_2,f_1\}_{\vec{\bm m}}\subset\{\Theta^n_{(E^{m_3}X_3,X_4)}(f_3),f_2,f_1\}_{(n+m_3,m_2,m_1)}\circ (1_{E^{|\vec{\bm m}|}X_1}\wedge\tau(\s^1,\s^n)).
$$
Conversely take $\beta=[\Theta^n_{(E^{m_3}X_3,X_4)}(f_3),B_2,E^nE^{m_3}f_2]\circ(E^nE^{m_3}f_2,\widetilde{E}^{n+m_3}B_1,E^nE^{m_{[3,2]}}f_1)\in\{\Theta^n_{(E^{m_3}X_3,X_4)}(f_3),f_2,f_1\}_{(n+m_3,m_2,m_1)}$ arbitrarily. 
Set 
$$
A_2=\Theta^{-n}_{(CE^{m_{[3,2]}}X_2,X_4)}(B_2\circ(1_{E^{m_{[3,2]}}X_2}\wedge\tau(I,\s^n))):E^{m_{[3,2]}}X_2\wedge I\to \Omega^n X_4.
$$ 
Then $A_2:f_3\circ E^{m_3}f_2\simeq *$ and $ev_{X_4}\circ\widetilde{E}^nA_2=B_2$. 
Hence 
\begin{align*}
\beta&=[\Theta^n_{(E^{m_3}X_3,X_4)}(f_3),B_2,E^nE^{m_3}f_2]\circ(E^nE^{m_3}f_2,\widetilde{E}^{n+m_3}B_1,E^nE^{m_{[3,2]}}f_1)\\
&=[\Theta^n_{(E^{m_3}X_3,X_4)}(f_3),ev_{X_4}\circ \widetilde{E}^nA_2,E^nE^{m_3}f_2]\circ(E^nE^{m_3}f_2,\widetilde{E}^{n+m_3}B_1,E^nE^{m_{[3,2]}}f_1)
\\
&=\Theta^n_{(EE^{|\vec{\bm m}|}X_1,X_4)}([f_3,A_2,E^{m_3}f_2]\circ(E^{m_3}f_2,\widetilde{E}^nB_1,E^{m_{[3,2]}}f_1))
\circ(1_{E^{|\vec{\bm m}|}X_1}\wedge\tau(\s^n,\s^1))\\
&\hspace{3cm}(\text{by the above discussion})
\\
&\in\Theta^n_{(EE^{|\vec{\bm m}|}X_1,X_4)}\{f_3,f_2,f_1\}_{\vec{\bm m}}\circ(1_{E^{m_{[3,1]}}X_1}\wedge\tau(\s^n,\s^1))
\end{align*}
and so $\{\Theta^n_{(E^{m_3}X_3,X_4)}(f_3),f_2,f_1\}_{n+m_3,m_2,m_1}\subset\Theta^n_{(EE^{|\vec{\bm m}|}X_1,X_4)}\{f_3,f_2,f_1\}_{\vec{\bm m}}\circ(1_{E^{|\vec{\bm m}|}X_1}\wedge\tau(\s^n,\s^1))$. 
Therefore we have (2.2.3). 

(3) (a) Since 
\begin{align*}
f_4\circ E^{m_4}f_3
&=\Theta_{(E^{m_4-1}X_4,X_5)}(f^*_4)\circ E^{m_4}f_3\\
&=\Theta_{(E^{m_4-1}X_4,X_5)}(\Theta^{-1}_{(E^{m_4-1}X_4,X_5)}(f_4))\circ EE^{m_4-1}f_3\\
&=\Theta_{(E^{m_4-1}E^{m_3}X_3,X_5)}(\Theta_{(E^{m_4-1}X_4,X_5)}^{-1}(f_4)\circ E^{m_4-1}f_3)\quad (\text{by (2.2.1)})\\
&=\Theta_{(E^{m_4-1}E^{m_3}X_3,X_5)}(f^*_4\circ E^{m_4-1}f_3),
\end{align*}
it follows that $f_4\circ E^{m_4}f_3\simeq *$ if and only if $f^*_4\circ E^{m_4-1}f_3\simeq *$. 
This proves (a). 

(b) Suppose that $(A_3,A_2,A_1)$ is an admissible triple for $(f_4,f_3,f_2,f_1)_{\vec{\bm m}}$. 
That is, $A_k:CE^{m_{k+1}}E^{m_k}X_k\to X_{k+2}$ is a null homotopy of $f_{k+1}\circ E^{m_{k+1}}f_k$ for $k=1,2,3$ such that 
\begin{gather*}
[f_4,A_3,E^{m_4}f_3]\circ(E^{m_4}f_3,\widetilde{E}^{m_4}A_2,E^{m_{[4,3]}}f_2)\simeq *,\\
[f_3,A_2,E^{m_3}f_2]\circ(E^{m_3}f_2,\widetilde{E}^{m_3}A_1,E^{m_{[3,2]}}f_1)\simeq *.
\end{gather*}
Set $B_3=\Theta_{(CE^{m_4-1}E^{m_3}X_3,X_5)}^{-1}(A_3\circ(1_{E^{m_4-1}E^{m_3}X_3}\wedge \tau(I,\s^1)):CE^{m_4-1}E^{m_3}X_3\to\Omega X_5$. 
Then $B_3:f^*_4\circ E^{m_4-1}f_3\simeq *$ by (1) for $\Omega X_5\overset{f^*_4}{\longleftarrow}E^{m_4-1}X_4\overset{E^{m_4-1}f_3}{\longleftarrow}E^{m_4-1}E^{m_3}X_3$. 
We will prove $\alpha:=[f^*_4,B_3,E^{m_4-1}f_3]\circ(E^{m_4-1}f_3,\widetilde{E}^{m_4-1}A_2,E^{m_4-1}E^{m_3}f_2)\simeq *$. 
Once we have proved it, $(B_3,A_2,A_1)$ is an admissible triple for $(f^*_4,f_3,f_2,f_1)_{(m_4-1,m_3,m_2,m_1)}$. 
Now 
\begin{align*}
&\Theta_{(EE^{m_4-1}E^{m_{[3,2]}}X_2,X_5)}(\alpha)\\
&=\Theta_{(E^{m_4-1}X_4\cup CE^{m_4-1}E^{m_3}X_3,X_5)}([f^*_4,B_3,E^{m_4-1}f_3])\\
&\hspace{3cm}\circ E(E^{m_4-1}f_3,\widetilde{E}^{m_4-1}A_2,E^{m_4-1}E^{m_3}f_2)
\quad(\text{by (2.2.1)})\\
&=[\Theta_{(E^{m_4-1}X_4,X_5)}(f^*_4),ev_{X_5}\circ\widetilde{E}B_3,EE^{m_4-1}f_3]\circ (\psi^1_{E^{m_4-1}f_3})^{-1}
\\
&\hspace{1cm}\circ\psi^1_{E^{m_4-1}f_3}\circ(E^{m_4}f_3,\widetilde{E}^{m_4}A_2,E^{m_{[4,3]}}f_2)
\circ(1_{E^{m_4-1}E^{m_{[3,2]}}X_2}\wedge\tau(\s^1,\s^1))
\\
&\hspace{3cm}(\text{by (2.2.2) and \cite[Lemma 2.4]{OO1}})\\
&=[f_4,A_3,E^{m_4}f_3]\circ
(E^{m_4}f_3,\widetilde{E}^{m_4}A_2,E^{m_{[4,3]}}f_2)\\
&\hspace{2cm}\circ(1_{E^{m_4-1}E^{m_{[3,2]}}X_2}\wedge\tau(\s^1,\s^1))\quad(\text{since $ev_{X_5}\circ\widetilde{E}B_3=A_3$})\\
&\simeq *.
\end{align*}
Hence $\alpha\simeq *$. 

Conversely suppose that $(B_3,B_2,B_1)$ is an admissible triple for $(f_4^*,f_3,f_2,f_1)_{(m_4-1,m_3,m_2,m_1)}$. 
Then $[f_4^*,B_3,E^{m_4-1}f_3]\circ(E^{m_4-1}f_3,\widetilde{E}^{m_4-1}B_2,E^{m_4-1}E^{m_3}f_2)\simeq *$ and
\begin{align*}
*&\simeq \Theta_{(EE^{m_4-1}E^{m_{[3,2]}}X_2,X_5)}([f_4^*,B_3,E^{m_4-1}f_3]\circ(E^{m_4-1}f_3,\widetilde{E}^{m_4-1}B_2,E^{m_4-1}E^{m_3}f_2))\\
&=\Theta_{(E^{m_4-1}X_4\cup_{E^{m_4-1}f_3}CE^{m_4-1}E^{m_3}X_3,X_5)}([f_4^*,B_3,E^{m_4-1}f_3])\\
&\hspace{3cm}\circ E(E^{m_4-1}f_3,\widetilde{E}^{m_4-1}B_2,E^{m_4-1}E^{m_3}f_2)\quad(\text{by (2.2.1)})\\
&=[f_4, ev_{X_5}\circ\widetilde{E}B_3,EE^{m_4-1}f_3]\circ(\psi^1_{E^{m_4-1}f_3})^{-1}\circ E(E^{m_4-1}f_3,\widetilde{E}^{m_4-1}B_2,E^{m_4-1}E^{m_3}f_2)\\
&\hspace{3cm}(\text{by (2.2.2)})\\
&=[f_4, ev_{X_5}\circ\widetilde{E}B_3,E^{m_4}f_3]\circ(E^{m_4}f_3,\widetilde{E}^{m_4}B_2,E^{m_{[4,3]}}f_2)
\circ(1_{E^{m_4-1}E^{m_{[3,2]}}X_2}\wedge\tau(\s^1,\s^1))
\\
&\hspace{3cm}(\text{by \cite[Lemma 2.4]{OO1}}).
\end{align*}
Hence $[f_4, ev_{X_5}\circ\widetilde{E}B_3,E^{m_4}f_3]\circ(E^{m_4}f_3,\widetilde{E}^{m_4}B_2,E^{m_{[4,3]}}f_2)\simeq *$ and 
$(ev_{X_5}\circ\widetilde{E}B_3,B_2,B_1)_{\vec{\bm m}}$ is an admissible triple for $(f_4,f_3,f_2,f_1)_{\vec{\bm m}}$. 
This proves (b). 

(c) Let $(B_3,B_2,B_1)$ be an admissible triple for $(f_4^*,f_3,f_2,f_1)$. 
We have  
\begin{align*}
&\{B_3,B_2,B_1\}^{(2)}_{(m_4-1,m_3,m_2,m_1)}\\
&=\{f_4^*,[f_3,B_2,E^{m_3}f_2],(E^{m_3}f_2,\widetilde{E}^{m_3}B_1,E^{m_{[3,2]}}f_1)\}_{m_4-1}\\
&\quad\cap\{[f_4^*,B_3,E^{m_4-1}f_3]\circ(\psi^{m_4-1}_{f_3})^{-1},(f_3,B_2,E^{m_3}f_2),-EE^{m_{[3,2]}}f_1\}_{m_4-1}
\end{align*}
by definition, and
\begin{align*}
&\mathrm{Indet}\{B_3,B_2,B_1\}^{(2)}_{(m_4-1,m_3,m_2,m_1)}\\
&\supset[E^2E^{m_4-1}E^{m_{[3,2]}}X_2,\Omega X_5]\circ E^2E^{m_4-1}E^{m_{[3,2]}}f_1+f_4^*\circ E^{m_4-1}[E^2E^{m_{[3,1]}}X_1,X_4]
\end{align*}
by (2.1.1). 
It follows from (2) that
\begin{align*}
&\Theta_{(EE^{m_4-1}EE^{m_{[3,1]}}X_1,X_5)}\{f_4^*,[f_3,B_2,E^{m_3}f_2],(E^{m_3}f_2,\widetilde{E}^{m_3}B_1,E^{m_{[3,2]}}f_1)\}_{m_4-1}\\
&=
\{f_4,[f_3,B_2,E^{m_3}f_2],(E^{m_3}f_2,\widetilde{E}^{m_3}B_1,E^{m_{[3,2]}}f_1)\}_{1+m_4-1}\circ(1_{E^{m_{[4,1]}}X_1}\wedge\tau(\s^1,\s^1)),\\
&\Theta_{(EE^{m_4-1}EE^{m_{[3,1]}}X_1,X_5)}\{[f_4^*,B_3,E^{m_4-1}f_3]\circ(\psi^{m_4-1}_{f_3})^{-1},(f_3,B_2,E^{m_3}f_2),-EE^{m_{[3,2]}}f_1\}_{m_4-1}\\
&=
\{\Theta_{(E^{m_4-1}(X_4\cup_{f_3}CE^{m_3}X_3),X_5)}([f_4^*,B_3,E^{m_4-1}f_3]\circ(\psi^{m_4-1}_{f_3})^{-1}),(f_3,B_2,E^{m_3}f_2),-EE^{m_{[3,2]}}f_1\}_{m_4}\\
&\hspace{3cm}\circ(1_{E^{m_4-1}EE^{m_{[3,1]}}X_1}
\wedge\tau(\s^1,\s^1))\quad(\text{by (2.2.3)})\\
&=\{[f_4,ev_{X_5}\circ\widetilde{E}B_3,E^{m_4}f_3]\circ(\psi^{m_4}_{f_3})^{-1},(f_3,B_2,E^{m_3}f_2),-EE^{m_{[3,2]}}f_1\}_{m_4}\circ(1_{E^{m_{[4,1]}}X_1}\wedge\tau(\s^1,\s^1)),
\end{align*}
where the last equality follows from the following equalities:
\begin{align*}
&\Theta_{(E^{m_4-1}(X_4\cup_{f_3}CE^{m_3}X_3),X_5)}([f_4^*,B_3,E^{m_4-1}f_3]\circ(\psi^{m_4-1}_{f_3})^{-1})\\
&=\Theta_{(E^{m_4-1}X_4\cup CE^{m_4-1}E^{m_3}X_3,X_5)}([f^*_4,B_3,E^{m_4-1}f_3])\circ E(\psi^{m_4-1}_{f_3})^{-1}\quad (\text{by (2.2.1)})\\
&=[\Theta_{(E^{m_4-1}X_4,X_5)}(f^*_4),ev_{X_5}\circ\widetilde{E}B_3,E^{m_4}f_3]\circ(\psi^1_{E^{m_4-1}f_3})^{-1}\circ E(\psi^{m_4-1}_{f_3})^{-1}\quad (\text{by (2.2.2)})\\
&=[f_4,ev_{X_5}\circ\widetilde{E}B_3,E^{m_4}f_3]\circ(\psi^{m_4}_{f_3})^{-1}.
\end{align*}
Hence
\begin{align*}
&\Theta_{(EE^{m_4-1}EE^{m_{[3,1]}}X_1,X_5)}\{B_3,B_2,B_1\}^{(2)}_{(m_4-1,m_3,m_2,m_1)}\\
&=\big(\{f_4,[f_3,B_2,E^{m_3}f_2],(E^{m_3}f_2,\widetilde{E}^{m_3}B_1,E^{m_{[3,2]}}f_1)\}_{1+m_4-1}\\
&\hspace{2cm}\cap
\{[f_4,ev_{X_5}\circ\widetilde{E}B_3,E^{m_4}f_3]\circ(\psi^{m_4}_{f_3})^{-1},(f_3,B_2,E^{m_3}f_2),-EE^{m_{[3,2]}}f_1\}_{m_4}\big)\\&\hspace{3cm}\circ(1_{E^{m_{[4,1]}}X_1}\wedge\tau(\s^1,\s^1))\\
&=\{ev_{X_5}\circ\widetilde{E}B_3,B_2,B_1\}^{(2)}_{(m_4,m_3,m_2,m_1)}\circ(1_{E^{m_{[4,1]}}X_1}\wedge\tau(\s^1,\s^1))\\
&\hspace{3cm}(\text{since $(ev_{X_5}\circ\widetilde{E}B_3,B_2,B_1)$ is an admissible triple for $\vec{\bm f}_{\vec{\bm m}}$ by (b)}).
\end{align*}
Hence $\Theta_{(EE^{m_4-1}EE^{m_{[3,1]}}X_1,X_5)}\{f_4^*,f_3,f_2,f_1\}^{(2)}_{(m_4-1,m_3,m_2,m_1)}\subset\{f_4,f_3,f_2,f_1\}_{(m_4,m_3,m_2,m_1)}^{(2)}\circ(1_{E^{m_{[4,1]}}X_1}\wedge\tau(\s^1,\s^1))$. 

Conversely let $(A_3,A_2,A_1)_{\vec{\bm m}}$ be an admissible triple for $\vec{\bm f}_{\vec{\bm m}}$. 
Set 
$$
B_3=\Theta_{(CE^{m_4-1}E^{m_3}X_3,X_5)}^{-1}(A_3\circ(1_{E^{m_4-1}E^{m_3}X_3}\wedge\tau(I,\s^1)):CE^{m_4-1}E^{m_3}X_3\to \Omega X_5.
$$ 
Then, as in (b), we have that $(B_3,A_2,A_1)$ is an admissible triple for $(f_4^*,f_3,f_2,f_1)_{(m_4-1,m_3,m_2,m_1)}$, 
$ev_{X_5}\circ\widetilde{E}B_3=A_3$, and  
\begin{align*}
&\Theta_{(E^2E^{m_4-1}E^{m_{[3,1]}}X_1,X_5)}\{B_3,A_2,A_1\}^{(2)}_{(m_4-1,m_3,m_2,m_1)}\\
&\hspace{3cm}=\{A_4,A_2,A_1\}_{\vec{\bm m}}^{(2)}\circ(1_{E^{m_{[4,1]}}X_1}\wedge\tau(\s^1,\s^1))\quad\text{and}\\ 
&\Theta_{(E^2E^{m_4-1}E^{m_{[3,1]}}X_1,X_5)}\{f_4^*,f_3,f_2,f_1\}^{(2)}_{(m_4-1,m_3,m_2,m_1)}\\
&\hspace{4cm}\supset\{f_4,f_3,f_2,f_1\}_{\vec{\bm m}}^{(2)}\circ(1_{E^{m_{[4,1]}}X_1}\wedge\tau(\s^1,\s^1)).
\end{align*} 
Therefore we obtain (2.2.4). 
This ends the proof of (c). 

(d) In the proof we use maps $i:\s^n\to \s^n_\infty$ and $\Phi_{\s^n}:\s^n_\infty\simeq \Omega E\s^n$ which shall be defined in the next subsection. 
By (3)(b), it suffices to prove that $(f_4^*,f_3,f_2,f_1)_{(m_4-1,m_3,m_2,m_1)}$ is admissible if and only if $(i\circ f'_4,f_3,f_2,f_1)_{(m_4-1,m_3,m_2,m_1)}$ is admissible. 
Note that $\Phi_{\s^n}\circ i\circ f'_4=f_4^*$. 

First let $(K_3,A_2,A_1)$ be an admissible triple for 
$(i\circ f'_4,f_3,f_2,f_1)$. 
Since $[i\circ f'_4,K_3,E^{m_4-1}f_3]\circ(E^{m_4-1}f_3,\widetilde{E}^{m_4-1}A_2,E^{m_4-1}E^{m_3}f_2)\simeq *$, we have
\begin{align*}
*&\simeq \Phi_{\s^n}\circ[i\circ f'_4,K_3,E^{m_4-1}f_3]\circ 
(E^{m_4-1}f_3,\widetilde{E}^{m_4-1}A_2,E^{m_4-1}E^{m_3}f_2)\\
&=[\Phi_{\s^n}\circ i\circ f'_4,\Phi_{\s^n}\circ K_3,E^{m_4-1}f_3]
\circ (E^{m_4-1}f_3,\widetilde{E}^{m_4-1}A_2,E^{m_4-1}E^{m_3}f_2).
\end{align*}
Hence $(\Phi_{\s^n}\circ K_3,A_2,A_1)$ is an admissible triple for $(f_4^*,f_3,f_2,f_1)$. 

Secondly let $(L_3,A_2,A_1)$ be an admissible triple for $(f_4^*,f_3,f_2,f_1)$. 
Let $\Phi^{-1}_{\s^n}$ be a homotopy inverse of $\Phi_{\s^n}$. 
Since $f^*_4=\Phi_{\s^n}\circ i\circ f'_4$, we have 
$\Phi^{-1}_{\s^n}\circ f^*_4\simeq i\circ f'_4$. 
Let $J:\Phi^{-1}_{\s^n}\circ f^*_4\simeq i\circ f'_4$ and set 
$M_3=(\Phi^{-1}_{\s^n}\circ L_3)\bullet\big((-J)\,\bar{\circ}\,\widetilde{E}^{m_4-1}(-1_{f_3})\big)$ (see \cite{OO1} for definitions of $\bullet$ and $\bar{\circ}$). 
Then $M_3:i\circ f'_4\circ E^{m_4-1}f_3\simeq *$ by \cite[Lemma 2.10]{OO1} and  
\begin{align*}
&[i\circ f'_4, M_3, E^{m_4-1}f_3]\circ(E^{m_4-1}f_3,\widetilde{E}^{m_4-1}A_2,E^{m_4-1}E^{m_3}f_2)\\
&\simeq [\Phi^{-1}_{\s^n}\circ f^*_4,\Phi^{-1}_{\s^n}\circ L_3,E^{m_4-1}f_3]\circ(E^{m_4-1}f_3,\widetilde{E}^{m_4-1}A_2,E^{m_4-1}E^{m_3}f_2)\\
&\hspace{4cm}(\text{by \cite[Lemma 2.10]{OO1}})\\
&=\Phi^{-1}_{\s^n}\circ[f_4^*,L_3,E^{m_4-1}f_3]\circ(E^{m_4-1}f_3,\widetilde{E}^{m_4-1}A_2,E^{m_4-1}E^{m_3}f_2)\\
&\simeq *.
\end{align*}
Hence $(M_3,A_2,A_1)$ is an admissible triple for $(i\circ f'_4,f_3,f_2,f_1)$. 
Thus (d) and (3) are proved. 
This completes the proof of Lemma 2.2.1.
\end{proof}

\subsection{Generalized Hopf invariant}

Let $Y$ be a {\it special complex}, namely, $Y$ is a countable CW-complex with only one vertex, say $y_0$, and let $Y_\infty$ be the {\it reduced product complex} of $Y$ \cite{J1}. 
By definition, $Y_\infty$ is algebraically a free semi group generated by $Y-\{y_0\}$ i.e. it is the set of all ordered finite sequences of elements in $Y-\{y_0\}$ and the empty sequence is the left and right identity which is denoted by $y_0$. 
Let $Y_m$ be the set of sequences in $Y-\{y_0\}$ which have not more than $m$ terms, with the identification topology determined by $p_m:Y^m\to Y_m,\ (a_1,\dots,a_m)\mapsto a_1\cdots a_m$, of course if $a_i=y_0$ then $p_m(a_1,\dots,a_m)=a_1\cdots a_{i-1}a_{i+1}\cdots a_m$. 
Thus $Y_0=\{y_0\}$, $Y_1=Y$, and $Y_m\subset Y_{m+1}$. 
As a topological space $Y_\infty=\underset{m\to\infty}{\lim} Y_m$. 
Then $Y_\infty$ is also a special complex. 
Let $i:Y\to Y_\infty$ and $i_Y:Y\to\Omega EY$ be the canonical injections, that is, $i:Y_1\subset Y_\infty$ and $i_Y(y)(\bar{t})=y\wedge \bar{t}$. 
Let $\Phi_Y:Y_\infty\to \Omega EY$ be the {\it canonical map} \cite{J1} which is the identity on $Y$, that is, $\Phi_Y\circ i=i_Y$, and so 
\begin{equation}
\Theta_{(Y,EY)}(\Phi_Y\circ i)=1_{EY}.
\end{equation} 
The map $\Phi_Y$ is 
a weak homotopy equivalence by \cite[(5.15)]{J1} and so it  
is a homotopy equivalence, since $EY$ is a CW-complex and $\Omega EY$ has a homotopy type of CW-complex by \cite{M}. 
For any space $X$ we have a bijection 
\begin{equation}
\phi_Y=\Theta_{(X,EY)}\circ{\Phi_Y}_*:[X,Y_\infty]\to  [EX,EY]\end{equation}
which is natural in $X$. 
Notice that if $X$ is a coH-space then $\phi_Y:[X,Y_\infty]\to[EX,EY]$ is an isomorphism of groups.  

The space $Y_2/Y_1=Y\wedge Y$ is a special complex again.  
Let $h:(Y_\infty,Y)\to ((Y\wedge Y)_\infty,*)$ be the combinatorial extension (\cite[(2.5)]{J1}) of
$$
(Y_2,Y_1)\overset{q}{\longrightarrow}(Y_2/Y_1,*)=(Y\wedge Y,*)\subset((Y\wedge Y)_\infty,*),
$$
where $q$ is the quotient map. 
When $Y=\s^n$ with $n\ge 1$, we denote $h$ by $h_n:(\s^n_\infty,\s^n)\to(\s^{2n}_\infty,*)$. 
For any space $X$ we define a {\it generalized Hopf invariant} $H$ by the formula
\begin{equation}
H=\phi_{\s^{2n}}\circ h_{n*}\circ\phi_{\s^n}^{-1}
:[EX,E\s^n]\overset{\phi^{-1}_{\s^n}}{\cong}[X,\s^n_\infty]\overset{h_{n*}}{\longrightarrow}[X,\s^{2n}_\infty]\overset{\phi_{\s^{2n}}}{\cong}[EX,E\s^{2n}].
\end{equation}

Notice that $H$ is natural in $X$ and that $H$ is a homomorphism of groups if $X$ is a coH-space. 

\begin{lemma'}
If a sequence $\vec{\bm f}=(f_4,f_3,f_2,f_1)$ of maps 
$X_5=\s^n\overset{f_4}{\longleftarrow}E^{m_4}X_4$ with $n\ge 1$, $X_{k+1}\overset{f_k}{\longleftarrow}E^{m_k}X_k\ (k=1,2,3)$  
satisfies 
\begin{gather}
E(f_4\circ E^{m_4}f_3)\simeq *,\ f_3\circ E^{m_3}f_2\simeq *,\ f_2\circ E^{m_2}f_1\simeq *,\\
\text{$\{Ef_4,f_3,f_2\}_{(1+m_4,m_3,m_2)}=\{0\}$ and $\{f_3,f_2,f_1\}_{(m_3,m_2,m_1)}\ni 0$},
\end{gather}
then $(i\circ f_4,f_3,f_2,f_1)_{\vec{\bm m}}$ is admissible, where $i:\s^n\to \s^n_\infty$ is the canonical inclusion. 
\end{lemma'}
\begin{proof}
Since 
\begin{align*}
&\Theta_{(E^{m_{[4,3]}}X_3,E\s^n)}(\Phi_{\s^n}\circ i\circ f_4\circ E^{m_4}f_3)\\
&=
\Theta_{(E^{m_4}X_4,E\s^n)}(\Phi_{\s^n}\circ i\circ f_4)\circ EE^{m_4}f_3\quad(\text{by (2.2.1)})\\
&=
Ef_4\circ EE^{m_4}f_3 \quad(\text{by (2.3.1)})\\
&\simeq * \quad(\text{by (2.3.4)}),
\end{align*}
it follows that $\Phi_{\s^n}\circ i\circ f_4\circ E^{m_4}f_3\simeq *$ 
so that $i\circ f_4\circ E^{m_4}f_3\simeq *$. 
We have 
\begin{align*}
&\Theta_{(EE^{m_{[4,2]}}X_2,E\s^n)}\{\Phi_{\s^n}\circ i\circ f_4,f_3,f_2\}_{(m_4,m_3,m_2)}\\
&=\{\Theta_{(EE^{m_4}X_4,E\s^n)}(\Phi_{\s^n}\circ i\circ f_4),f_3,f_2\}_{(1+m_4,m_3,m_2)}\circ(1_{E^{m_{[4,2]}}X_2}\wedge\tau(\s^1,\s^1))\quad(\text{by (2.2.3)})\\
&=\{Ef_4,f_3,f_2\}_{(1+m_4,m_3,m_2)}\circ(1_{E^{m_{[4,2]}}X_2}\wedge\tau(\s^1,\s^1))\quad(\text{by (2.3.1)})\\
&=\{0\} \quad(\text{by (2.3.5)})
\end{align*}
and so $\{\Phi_{\s^n}\circ i\circ f_4,f_3,f_2\}_{(m_4,m_3,m_2)}=\{0\}$. 
Since ${\Phi_{\s^n}}_*:[E^{m_{[4,2]}}X_2,\s^n_\infty]\to [E^{m_{[4,2]}}X_2,\Omega E\s^n]$ is an isomorphism and 
${\Phi_{\s^n}}_*\{i\circ f_4, f_3, f_2\}_{(m_4,m_3,m_2)}\subset\{\Phi_{\s^n}\circ i\circ f_4,f_3,f_2\}_{(m_4,m_3,m_2)}=\{0\}$, we have 
$\{i\circ f_4,f_3,f_2\}_{(m_4,m_3,m_2)}=\{0\}$. 
Take null homotopies $A_2:f_3\circ E^{m_3}f_2\simeq *$ and $A_1:f_2\circ E^{m_2}f_1\simeq *$ such that $[f_3,A_2,E^{m_3}f_2]\circ(E^{m_3}f_2,\widetilde{E}^{m_3}A_1,E^{m_{[3,2]}}f_1)\simeq *$. This is possible by (2.3.5). 
Take a null homotopy $A_3:i\circ f_4\circ E^{m_4}f_3\simeq *$ arbitrarily. 
Then $(A_3,A_2,A_1)$ is an admissible triple for $(i\circ f_4,f_3,f_2,f_1)_{\vec{\bm m}}$. This completes the proof.
\end{proof}

\subsection{Proofs of Theorem 1.1 and Corollary 1.2}

Consider the following commutative diagram for $n\ge 1$, where the first row is the exact sequence for the pair $(\s^n_\infty,\s^n)$ and $q$ is the quotient. 
$$
\xymatrix{
\cdots \ar[r] &\pi_r(\s^n) \ar[r] \ar[dr]_-E & \pi_r(\s^n_\infty) \ar[r] \ar[dr]_-{h_{n*}} \ar[d]_-\cong^-{\phi_{\s^n}} & \pi_r(\s^n_\infty,\s^n) \ar[r]^-{\partial} \ar[d]^-{h_{n*}}&\pi_{r-1}(\s^n) \ar[r] \ar@{->>}[d]^-q &\cdots\\
& & \pi_{r+1}(\s^{n+1}) \ar[dr]_-H & \pi_r(\s^{2n}_\infty)  \ar[d]_-\cong^-{\phi_{\s^{2n}}} &\pi_{r-1}(\s^n)/\partial \mathrm{Ker}(h_{n*})& \\
&&&\pi_{r+1}(\s^{2n+1}) &&
}
$$
Define a homomorphism $\widetilde{\Delta}_n$ from a subgroup of $\pi_{r+1}(\s^{2n+1})$ to $\pi_{r-1}(\s^n)/\partial\mathrm{Ker}(h_{n*})$ by
$$
\widetilde{\Delta}_n=q\circ\partial\circ h_{n*}^{-1}\circ\phi_{\s^{2n}}^{-1}:\mathrm{Image}(\phi_{\s^{2n}}\circ h_{n*})\to \pi_{r-1}(\s^n)/\partial\mathrm{Ker}(h_{n*}).
$$
The next is a key from which Theorem 1.1 and Corollary 1.2 shall follow.  

\begin{prop'}
Given a sequence $\vec{\bm f}_{\vec{\bm m}}$ of maps $X_{k+1}\overset{f_k}{\longleftarrow}E^{m_k}X_k\ (1\le k\le 4)$ such that $X_5=\s^{n_5}\ (n_5\ge 1)$, $X_3=\s^{n_3}\ (n_3\ge 1)$, and 
$(i\circ f_4,f_3,f_2,f_2,f_1)_{\vec{\bm m}}$ is admissible, where $i:\s^{n_5}\to\s^{n_5}_\infty$ is the inclusion, we have 
$(Ef_4,f_3,f_2,f_1)_{(1+m_4,m_3,m_2,m_1)}$ is admissible and 
\begin{equation}
\begin{split}
H\{Ef_4,f_3,f_2,&f_1\}^{(2)}_{(1+m_4,m_3,m_2,m_1)}\\
&\subset\bigcup_{\alpha\in\widetilde{\Delta}_{n_5}^{-1}([f_4\circ E^{m_4}f_3])}(-1)^{m_4}\{\alpha,EE^{m_3}f_2,EE^{m_{[3,2]}}f_1\}_{1+m_4},
\end{split}
\end{equation}
where $H=\phi_{\s^{2n_5}}\circ{h_{n_5}}_*\circ\phi_{\s^{n_5}}^{-1} 
: [E^{|\vec{\bm m}|+3}X_1,E\s^{n_5}]\to [E^{|\vec{\bm m}|+3}X_1,E\s^{2n_5}]$ 
is the generalized Hopf invariant (2.3.3) and 
$[f_4\circ E^{m_4}f_3]\in\pi_{n_3+m_3+m_4}(\s^{n_5})/\partial\mathrm{Ker}(h_{n_5*})$ is the class represented by $f_4\circ E^{m_4}f_3$. 
\end{prop'}

\begin{proof} 
We should emphasize that, as recalled in \S2.2, we use special identifications
$$
\s^n=\underbrace{\s^1\wedge\cdots\wedge\s^1}_n,\quad \s^{n_1}\wedge\s^{n_2}\wedge\cdots\wedge\s^{n_k}=\s^{n_1+\dots+n_k}.
$$

Note that $Ef_4=\Theta_{(E^{m_4}X_4,E\s^{n_5})}(\Phi_{\s^{n_5}}\circ i\circ f_4)$. 
Let $(A_3,A_2,A_1)$ be an admissible triple for 
$(i\circ f_4,f_3,f_2,f_1)_{\vec{\bm m}}$. 
Then, as is easily seen, 
\begin{equation}
\text{$(\Phi_{\s^{n_5}}\circ A_3,A_2,A_1)$ is an admissible triple for $(\Phi_{\s^{n_5}}\circ i\circ f_4,f_3,f_2,f_1)_{\vec{\bm m}}$}.
\end{equation}
Also $(ev_{E\s^{n_5}}\circ\widetilde{E}(\Phi_{\s^{n_5}}\circ A_3),A_2,A_1)$ is an admissible triple for $(Ef_4,f_3,f_2,f_1)$. 
Indeed we have
\begin{align*}
*&\simeq\Theta_{(E^{m_{[4,2]}}X_2,E\s^{n_5})}\big(\Phi_{\s^{n_5}}\circ[i\circ f_4,A_3,E^{m_4}f_3]\circ(E^{m_4}f_3,\widetilde{E}^{m_4}A_2,E^{m_{[4,3]}}f_2)\big)\\
&=\Theta_{(E^{m_{[4,2]}}X_2,E\s^{n_5})}\big([\Phi_{\s^{n_5}}\circ i\circ f_4,\Phi_{\s^{n_5}}\circ A_3,E^{m_4}f_3]\circ(E^{m_4}f_3,\widetilde{E}^{m_4}A_2,E^{m_{[4,3]}}f_2)\big)\\
&=\Theta_{(E^{m_4}X_4\cup CE^{m_{[4,3]}}X_3,E\s^{n_5})}([\Phi_{\s^{n_5}}\circ i\circ f_4,\Phi_{\s^{n_5}}\circ A_3,E^{m_4}f_3])\\
&\hspace{2cm}\circ E(E^{m_4}f_3,\widetilde{E}^{m_4}A_2,E^{m_{[4,3]}}f_2)\qquad (\text{by (2.2.1)})\\
&=[\Theta_{(E^{m_4}X_4,E\s^{n_5})}(\Phi_{\s^{n_5}}\circ i\circ f_4),ev_{E\s^{n_5}}\circ\widetilde{E}(\Phi_{\s^{n_5}}\circ A_3),EE^{m_4}f_3]\\
&\hspace{2cm}\circ (EE^{m_4}f_3,\widetilde{E}\widetilde{E}^{m_4}A_2,EE^{m_{[4,3]}}f_2)\circ(1_{E^{m_{[4,2]}}X_2}\wedge\tau(\s^1,\s^1))\\
&\hspace{4cm}(\text{by (2.2.2) and \cite[Lemma 2.4]{OO1}})
\end{align*}
and so $[Ef_4,ev_{E\s^{n_5}}\circ\widetilde{E}(\Phi_{\s^{n_5}}\circ A_3),EE^{m_4}f_3]\circ(E^{1+m_4}f_3,\widetilde{E}^{1+m_4}A_2,E^{1+m_4}E^{m_3}f_2)\simeq *$. 

We have the following homotopy commutative diagram. 
$$
\xymatrix{
\s^{n_5}_\infty \ar[d]_-{h_{n_5}} & &E^{m_4}X_4\cup_{E^{m_4}f_3}CE^{m_{[4,3]}}X_3 \ar[ll]_-{[i\circ f_4,A_3,E^{m_4}f_3]} \ar[d]^-{q_{E^{m_4}f_3}} & \ar[l]  \\
\s^{2n_5}_\infty & &EE^{m_{[4,3]}}X_3 \ar[ll]_-{g_{A_3}} & \ar[l]
}
$$
$$
\xymatrix{
& &&EE^{m_{[4,2]}}X_2 \ar[d]^-= \ar@{-}[lll]_-{(E^{m_4}f_3,\widetilde{E}^{m_4}A_2,E^{m_{[4,3]}}f_2)} & & EE^{m_{[4,1]}}X_1 \ar[ll]_-{-EE^{m_{[4,2]}}f_1} \ar[d]^-=\\
&&& EE^{m_{[4,2]}}X_2 \ar@{-}[lll]_-{-EE^{m_{[4,3]}}f_2} & &EE^{m_{[4,1]}}X_1 \ar[ll]_-{-EE^{m_{[4,2]}}f_1}
}
$$
where $g_{A_3}$ is defined to make the first square commutative. 
We have
\begin{align*}
{h_{n_5}}_*\{A_3,A_2,A_1\}^{(2)}_{\vec{\bm m}}&\subset {h_{n_5}}_*\{[i\circ f_4,A_3,E^{m_4}f_3]\circ(\psi^{m_4}_{f_3})^{-1},(f_3,A_2,E^{m_3}f_2),-EE^{m_{[3,2]}}f_1\}_{m_4}\\
&\subset\{{h_{n_5}}\circ[i\circ f_4,A_3,E^{m_4}f_3]\circ(\psi^{m_4}_{f_3})^{-1},(f_3,A_2,E^{m_3}f_2),-EE^{m_{[3,2]}}f_1\}_{m_4}\\
&=\{g_{A_3}\circ q_{E^{m_4}f_3}\circ(\psi^{m_4}_{f_3})^{-1},(f_3,A_2,E^{m_3}f_2),-EE^{m_{[3,2]}}f_1\}_{m_4}\\
&=\{g_{A_3}\circ(1_{E^{m_3}X_3}\wedge\tau(\s^1,\s^{m_4}))\circ E^{m_4}q_{f_3},(f_3,A_2,E^{m_3}f_2),-EE^{m_{[3,2]}}f_1\}_{m_4}\\ 
&\subset\{g_{A_3}\circ(1_{E^{m_3}X_3}\wedge\tau(\s^1,\s^{m_4})),q_{f_3}\circ(f_3,A_2,E^{m_3}f_2),-EE^{m_{[3,2]}}f_1\}_{m_4}\\
&=\{g_{A_3}\circ(1_{E^{m_3}X_3}\wedge\tau(\s^1,\s^{m_4})),-EE^{m_3}f_2,-EE^{m_{[3,2]}}f_1\}_{m_4}\\
&=\{g_{A_3}\circ(1_{E^{m_3}X_3}\wedge\tau(\s^1,\s^{m_4})),EE^{m_3}f_2,EE^{m_{[3,2]}}f_1\}_{m_4}\\
&=\{g_{A_3}\circ(1_{E^{m_3}X_3}\wedge\tau(\s^1,\s^{m_4})),E^{m_3}E f_2,E^{m_3}EE^{m_2}f_1\}_{m_4}\\
&\overset{*}{=}\{g_{A_3}\circ E^{m_4}E^{m_3}E((-1)^{m_4}1_{X_3}),E^{m_3}E f_2,E^{m_3}EE^{m_2}f_1\}_{m_4}\\
&\overset{**}{=}\{g_{A_3},E^{m_3}E((-1)^{m_4}1_{X_3})\circ E^{m_3}Ef_2,E^{m_3}EE^{m_2}f_1\}_{m_4}\\
&=\{g_{A_3},E^{m_3}Ef_2\circ E^{m_3}((-1)^{m_4}1_{EE^{m_2}X_2}),E^{m_3}EE^{m_2}f_1\}_{m_4}\\
&\overset{***}{=}\{g_{A_3},E^{m_3}Ef_2,E^{m_3}((-1)^{m_4}1_{EE^{m_2}X_2})\circ E^{m_3}EE^{m_2}f_1\}_{m_4}\\
&=\{g_{A_3},E^{m_3}Ef_2,(-1)^{m_4}E^{m_3}EE^{m_2}f_1\}_{m_4}\\
&=(-1)^{m_4}\{g_{A_3},E^{m_3}Ef_2,E^{m_3}EE^{m_2}f_1\}_{m_4},
\end{align*}
where the equality $\overset{*}{=}$ holds since $1_{E^{m_3}X_3}\wedge\tau(\s^1,\s^{m_4})$ is a self homeomorphism of $\s^{n_3+m_3+m_4+1}$ with the degree $(-1)^{m_4}$ so that 
it is homotopic to $E^{m_4}E^{m_3}E((-1)^{m_4}1_{X_3})$,  
the equality $\overset{**}{=}$ holds since $E^{m_3}E((-1)^{m_4}1_{X_3})$ is a homeomorphism, and  
the equality $\overset{***}{=}$ holds since $E^{m_3}((-1)^{m_4}1_{EE^{m_2}X_2})$ is a homeomorphism. 
Hence 
\begin{equation}
{h_{n_5}}_*\{i\circ f_4,f_3,f_2,f_1\}^{(2)}_{\vec{\bm m}}\subset\bigcup(-1)^{m_4}\{g_{A_3},E^{m_3}Ef_2,E^{m_3}EE^{m_2}f_1\}_{m_4},
\end{equation}
where the union $\bigcup$ is taken over all admissible triple $\vec{\bm A}=(A_3,A_2,A_1)$ for $(i\circ f_4,f_3,f_2,f_1)$. 
By Lemma 2.1.1, we have 
$$
\Phi_{\s^{n_5}}\circ\{i\circ f_4,f_3,f_2,f_1\}^{(2)}_{\vec{\bm m}}
=\{\Phi_{\s^{n_5}}\circ i\circ f_4,f_3,f_2,f_1\}^{(2)}_{\vec{\bm m}}.
$$
Consider the following three self maps of $\s^{|\vec{\bm m}|+3}=\s^{m_1+m_2+m_3+m_4+3}$: 
$$
1_{\s^{|\vec{\bm m}|}\wedge\s^1}\wedge\tau(\s^1,\s^1),\quad  
1_{\s^{|\vec{\bm m}|}}\wedge\tau(\s^1,\s^1)\wedge 1_{\s^1},\quad 1_{\s^{m_{[3,1]}}\wedge\s^1\wedge\s^{m_4}}\wedge\tau(\s^1,\s^1).
$$
Since their degrees are $-1$, they are homotopic each other and 
the composition of any two of them is homotopic to the identity. 
By smashing $1_{X_1}$ to the above three maps from left, we have three self homeomorphisms of $E^{|\vec{\bm m}|+3}X_1$:
\begin{gather*}
\xi=1_{EE^{|\vec{\bm m}|}X_1}\wedge\tau(\s^1,\s^1),\quad
\xi'=E(1_{E^{|\vec{\bm m}|}X_1}\wedge \tau(\s^1,\s^1)),\\
\xi''=1_{E^{m_4}EE^{m_{[3,1]}}X_1}\wedge\tau(\s^1,\s^1).
\end{gather*} 
They satisfy $\xi\simeq\xi'$ and $\xi'\circ\xi'\simeq \xi''\circ\xi'\simeq 1_{E^{|\vec{\bm m}|+3}X_1}$.  
We have
\begin{align*}
&\phi_{\s^{n_5}}\{i\circ f_4,f_3,f_2,f_1\}^{(2)}_{\vec{\bm m}}=\Theta_{(E^2E^{m_{[4,1]}}X_1,E\s^{n_5})}\Phi_{\s^{n_5}}\{i\circ f_4,f_3,f_2,f_1\}^{(2)}_{\vec{\bm m}}\\
&=\Theta_{(E^2E^{m_{[4,1]}}X_1,E\s^{n_5})}\{\Phi_{\s^{n_5}}\circ i\circ f_4,f_3,f_2,f_1\}^{(2)}_{\vec{\bm m}}\\
&=\{Ef_4,f_3,f_2,f_1\}^{(2)}_{(1+m_4,m_3,m_2,m_1)}\circ\xi\\
&\quad(\text{by Lemma 2.2.1(3)(c) for $(\Theta_{(E^{m_4}X_4,E\s^{n_5})}(\Phi_{\s^{n_5}}\circ i\circ f_4),f_3,f_2,f_1)_{(1+m_4,m_3,m_2,m_1)}$})\\
&=\xi'^*\{Ef_4,f_3,f_2,f_1\}^{(2)}_{(1+m_4,m_3,m_2,m_1)}\quad(\text{since $\xi\simeq\xi'$})
\end{align*}
so that $\{Ef_4,f_3,f_2,f_1\}^{(2)}_{(1+m_4,m_3,m_2,m_1)}=(\xi'^*\circ\phi_{\s^{n_5}})\{i\circ f_4,f_3,f_2,f_1\}^{(2)}_{\vec{\bm m}}$. 
Hence
\begin{align*}
H&\{Ef_4,f_3,f_2,f_1\}^{(2)}_{(1+m_4,m_3,m_2,m_1)}
=(H\circ\xi'^*\circ\phi_{\s^{n_5}})\{i\circ f_4,f_3,f_2,f_1\}^{(2)}_{\vec{\bm m}}\\
&=(\xi'^*\circ H\circ\phi_{\s^{n_5}})\{i\circ  f_4,f_3,f_2,f_1\}^{(2)}_{\vec{\bm m}}\quad(\text{since $\xi'^*\circ H=H\circ\xi'^*$})\\
&=(\xi'^*\circ\phi_{\s^{2n_5}}\circ{h_{n_5}}_*)\{i\circ f_4,f_3,f_2,f_1\}^{(2)}_{\vec{\bm m}}\quad(\text{by (2.3.3)})\\
&\subset(\xi'^*\circ\phi_{\s^{2n_5}})\bigcup_{\vec{\bm A}}(-1)^{m_4}
\{g_{A_3},EE^{m_3}f_2,EE^{m_{[3,2]}}f_1\}_{m_4}
\qquad(\text{by (2.4.3)})\\
&=\xi'^*\bigcup_{\vec{\bm A}}(-1)^{m_4}\{\Theta_{(E^{m_4}EE^{m_3}X_3, E\s^{2n_5})}(\Phi_{\s^{2n_5}}\circ g_{A_3}),\\
&\hspace{4cm} EE^{m_3}f_2,EE^{m_{[3,2]}}f_1\}_{1+m_4}\circ\xi''\qquad 
(\text{by Lemma 2.2.1(2)})\\
&=\bigcup_{\vec{\bm A}}(-1)^{m_4}\{\Theta_{(E^{m_4}EE^{m_3}X_3, E\s^{2n_5})}(\Phi_{\s^{2n_5}}\circ g_{A_3}),
EE^{m_3}f_2,EE^{m_{[3,2]}}f_1\}_{1+m_4}\\
&\hspace{4cm}(\text{since $\xi''\circ\xi'\simeq 1_{E^{|\vec{\bm m}|+3}X_1}$})\\
&=\bigcup_{\vec{\bm A}}(-1)^{m_4}\{\phi_{\s^{2n_5}}(g_{A_3}),EE^{m_3}f_2,EE^{m_{[3,2]}}f_1\}_{1+m_4},
\end{align*}
that is, 
\begin{equation}
\begin{split}
H\{Ef_4,&f_3,f_2,f_1\}^{(2)}_{(1+m_4,m_3,m_2,m_1)}\\
&\subset\bigcup_{\vec{\bm A}}(-1)^{m_4}\{\phi_{\s^{2n_5}}(g_{A_3}), EE^{m_3}f_2,EE^{m_{[3,2]}}f_1\}_{1+m_4},
\end{split}
\end{equation}
where the union $\bigcup_{\vec{\bm A}}$ is taken over all admissible triples $\vec{\bm A}=(A_3,A_2,A_1)$ for $(i\circ f_4,f_3,f_2,f_1)$. 

Let $F:(CE^{m_{[4,3]}}\s^{n_3},E^{m_{[4,3]}}\s^{n_3},*)\to (E^{m_4}X_4\cup_{E^{m_4}f_3}CE^{m_{[4,3]}}\s^{n_3},E^{m_4}X_4,*)$ be the characteristic map. 
Then the element 
$$
[i\circ f_4,A_3,E^{m_4}f_3]\circ F\in[(CE^{m_{[4,3]}}\s^{n_3},E^{m_{[4,3]}}\s^{n_3},*),(\s^{n_5}_\infty,\s^{n_5},*)]=\pi_{n_3+m_3+m_4+1}(\s^{n_5}_\infty,\s^{n_5})
$$
satisfies 
$$
\partial([i\circ f_4,A_3,E^{m_4}f_3]\circ F)=f_4\circ E^{m_4}f_3\in\pi_{n_3+m_3+m_4}(\s^{n_5})
$$
and we have the following commutative diagram:
$$
\xymatrix{
& (CE^{m_{[4,3]}}\s^{n_3},E^{m_{[4,3]}}\s^{n_3}) \ar[d]^-F \ar[dl]_-{q_{1_{E^{m_{[4,3]}}\s^{n_3}}}}\\
(EE^{m_{[4,3]}}\s^{n_3},*) \ar[d]_-{g_{A_3}} & E^{m_4}X_4\cup_{E^{m_4}f_3}CE^{m_{[4,3]}}\s^{n_3},E^{m_4}X_4)\ar[l]^-{q_{E^{m_4}f_3}} \ar[d]^-{[i\circ f_4,A_3,E^{m_4}f_3]}\\
(\s^{2n_5}_\infty,*) & (\s^{n_5}_\infty,\s^{n_5}) \ar[l]_-{h_{n_5}}
}
$$
Hence 
$h_{n_5}\circ[i\circ f_4,A_3,E^{m_4}f_3]\circ F=g_{A_3}\circ q_{1_{E^{m_{[4,3]}}\s^{n_3}}}=g_{A_3}$, where the last equality follows from the identification $q_{1_{E^{m_{[4,3]}}\s^{n_3}}}=1_{EE^{m_{[4,3]}}\s^{n_3}}$ in 
$$
[(CE^{m_{[4,3]}}\s^{n_3},E^{m_{[4,3]}}\s^{n_3}),(EE^{m_{[4,3]}}\s^{n_3},*)]=[EE^{m_{[4,3]}}\s^{n_3},EE^{m_{[4,3]}}\s^{n_3}].
$$ 
We then have
\begin{align*}
\widetilde{\Delta}_{n_5}(\phi_{\s^{2n_5}}(g_{A_3}))&=p\circ\partial([i\circ f_4,A_3,E^{m_4}f_3]\circ F)\\
&=p(f_4\circ E^{m_4}f_3)=[f_4\circ E^{m_4}f_3]\in\pi_{n_3+m_3+m_4}(\s^{n_5})/\partial\mathrm{Ker}(h_{n_5*}),
\end{align*}
where $p:\pi_{n_3+m_3+m_4}(\s^{n_5})\to\pi_{n_3+m_3+m_4}(\s^{n_5})/\partial\mathrm{Ker}(h_{n_5*})$ is the quotient, and so 
$$
\phi_{\s^{2n_5}}(g_{A_3})\in\widetilde{\Delta}_{n_5}^{-1}([f_4\circ E^{m_4}f_3]).
$$ 
Thus (2.4.1) holds by (2.4.4). 
This completes the proof of Proposition~2.4.1.
\end{proof}

\begin{proof}[Proof of Theorem 1.1] 
Under assumptions of Theorem 1.1, $(i\circ f_4,f_3,f_2,f_1)$ 
is admissible by Lemma 2.3.1 and so Theorem 1.1 holds by Proposition 2.4.1. 
\end{proof}

\begin{proof}[Proof of Corollary 1.2] 
From the homotopy exact sequence for the pair $(\s^m_\infty,\s^m)$, Toda \cite[(2.11)]{T} considered the $EH\Delta$-sequence 
$$
\begin{CD}
\cdots @>\Delta>>\pi_i(\s^m)@>E>>\pi_{i+1}(\s^{m+1})@>H>>\pi_{i+1}(\s^{2m+1})@>\Delta>>\pi_{i-1}(\s^m)@>E>>\cdots
\end{CD}
$$
in accordance with the convention in the first paragraph of \cite[p.22]{T}. 
In particular we have the following exact sequence
\begin{equation}
\begin{CD}
\cdots @>\Delta>>\pi_i^m @>E>>\pi_{i+1}^{m+1}@>H>>\pi_{i+1}^{2m+1}@>\Delta>>\pi_{i-1}^m @>E>>\cdots 
\end{CD}
\end{equation}
where $\pi^m_i$ is the subgroup of $\pi_i(\s^m)$ defined by Toda \cite[(4.3)]{T} which satisfies $\pi^m_i\cong\pi_i(\s^m)/T$, where 
$T$ is the finite subgroup consisting of all elements of odd order. 

Consider the following sequence: 
$$
\begin{CD}
\s^4@<\nu_4<<\s^7@<2\nu_7<<\s^{10}@<\eta_{10}\circ\sigma_{11}<<\s^{18}@<\sigma_{18}<<\s^{25}.
\end{CD}
$$
First we will prove that this satisfies the assumptions of Theorem 1.1. 
(1.3) is obvious. 
Note that $\eta_n\circ\sigma_{n+1}=\varepsilon_{n}+\overline{\nu}_n$ for $n\ge 9$ by \cite[Theorem 7.1]{T} and that $\varepsilon_n\circ\sigma_{n+8}=0\ (n\ge 3)$ and $\overline{\nu}_n\circ\sigma_{n+8}=0\ (n\ge 6)$ by \cite[Lemma 10.7]{T}.  
As a consequence $(E\nu_4,2\nu_7,\eta_{10}\circ\sigma_{11},\sigma_{18})_{1,0}$ is a null quadruple, i.e. (1.1) holds. 
We will prove 
\begin{equation}
\{\nu_5, 2\nu_8, \eta_{11}\circ\sigma_{12}\}=\{0\},\quad \{2\nu_7,\eta_{10}\circ\sigma_{11},\sigma_{18}\}=\{0\}.
\end{equation}
Once we prove (2.4.6), then (1.2) holds. 
We have 
\begin{align*}
\{\nu_5,2\nu_8,\eta_{11}\circ\sigma_{12}\}&\subset\{\nu_5, 2\nu_8\circ\eta_{11},\sigma_{12}\}\\
&=\pi^5_{13}\circ\sigma_{13}+\nu_5\circ\pi^8_{20}\quad(\text{since $2\nu_8\circ\eta_{11}=0$})\\
&=\{0\}
\end{align*}
so that $\{\nu_5,2\nu_8,\eta_{11}\circ\sigma_{12}\}=\{0\}$. 
We have $\bar{\sigma}_6\in\{\nu_6,\varepsilon_8+\bar{\nu}_8,\sigma_{16}\}_1$ by \cite[p.138]{T}. 
Hence $\bar{\sigma}_7\in E\{\nu_6,\varepsilon_8+\bar{\nu}_8,\sigma_{16}\}_1=-\{\nu_7,\varepsilon_9+\bar{\nu}_9,\sigma_{17}\}_1$. 
Since the order of $\bar{\sigma}_7$ is $2$, we have $\bar{\sigma}_7\in\{\nu_7,\varepsilon_9+\bar{\nu}_9, \sigma_{17}\}_1$. 
Hence 
\begin{align*}
0&=2\bar{\sigma}_7\\
&=2\iota_7\circ\bar{\sigma}_7\quad(\text{since $\s^7$ is an H-space})\\
&\in 2\iota_7\circ\{\nu_7,\varepsilon_9+\bar{\nu}_9, \sigma_{17}\}_1\subset\{2\iota_7\circ \nu_7,\varepsilon_9+\bar{\nu}_9, \sigma_{17}\}_1=\{2\nu_7,\varepsilon_9+\bar{\nu}_9, \sigma_{17}\}_1.
\end{align*}
We have 
$$
\mathrm{Indet}\{2\nu_7,\varepsilon_9+\bar{\nu}_9, \sigma_{17}\}_1
=\pi^7_{19}\circ\sigma_{19}+2\nu_7\circ E\pi^9_{25}=\{0\}
$$
so that $\{2\nu_7,\varepsilon_9+\bar{\nu}_9, \sigma_{17}\}_1=\{0\}$, that is, $\{2\nu_7,\eta_9\circ\sigma_{10},\sigma_{17}\}_1=\{0\}$. 
Hence $0\in\{2\nu_7,\eta_{10}\circ\sigma_{11},\sigma_{18}\}$ and so 
\begin{align*}
\{2\nu_7,\eta_{10}\circ\sigma_{11},\sigma_{18}\}&=\mathrm{Indet}\{2\nu_7,\eta_{10}\circ\sigma_{11},\sigma_{18}\}=\pi^7_{19}\circ\sigma_{19}+2\nu_7\circ\pi^{10}_{26}\\
&=\{0\}\quad(\text{by \cite[Theorem 12.9]{T}}).
\end{align*}
Thus (2.4.6) is proved and so (1.2) holds. 
Therefore 
\begin{equation}
H\{E\nu_4, 2\nu_7,\eta_{10}\circ\sigma_{11},\sigma_{18}\}_{(1,0,0,0)}^{(2)}\subset
\bigcup_{\alpha\in\widetilde{\Delta}_4^{-1}([2\nu_4^2])}\{\alpha,\eta_{11}\circ\sigma_{12},\sigma_{19}\}_1
\end{equation}
by Theorem 1.1. 
Recall from \cite[Theorem 2.4]{T} that $h_4{*}:\pi_{11}(\s^4_\infty,\s^4)\to\pi_{11}(\s^8_\infty)$ is an isomorphism of $2$-primary components. 
Hence $2$-components of $\widetilde{\Delta}_4^{-1}([2\nu_4^2])$ and $\Delta^{-1}(2\nu_4^2)$ are the same, where $\Delta$ is the homomorphism of (2.4.5). 
Therefore the right hand term of (2.4.7) is $\bigcup_{\alpha\in\Delta^{-1}(2\nu_4^2)}\{\alpha,\eta_{11}\circ\sigma_{12},\sigma_{19}\}_1$, since $\eta_{11}\circ\sigma_{12}$ and $\sigma_{19}$ are $2$-primary elements. 
Consider the exact sequence (2.4.5): 
$$
\begin{CD}
\cdots@>H>>\pi^9_{12}=\bZ_8\{\nu_9\}@>\Delta>>\pi^4_{10}=\bZ_8\{\nu_4^2\}@>E>>\pi^5_{11}=\bZ_2\{\nu_5^2\}@>H>>\cdots.
\end{CD}
$$
By \cite[(5.13)]{T}, $\Delta(\nu_9)=\pm 2\nu_4^2$. 
We have 
$$
\Delta^{-1}(2\nu_4^2)=\begin{cases} \{\nu_9,5\nu_9\} & \Delta(\nu_9)=2\nu_4^2\\
\{3\nu_9,7\nu_9\} & \Delta(\nu_9)=-2\nu_4^2\end{cases}.
$$ 
We have $\mathrm{Indet}\{\nu_9,\eta_{11}\circ\sigma_{12},\sigma_{19}\}_1=\pi^9_{21}\circ\sigma_{21}+\nu_9\circ E\pi^{11}_{27}=\{0\}$ so that $\bar{\sigma}_9=E^3\bar{\sigma}_6=\{\nu_9,\eta_{11}\circ\sigma_{12},\sigma_{19}\}_1$. 
As is easily seen, $\{k\nu_9,\eta_{11}\circ\sigma_{12},\sigma_{19}\}_1=\bar{\sigma}_9$ for $k=1,3,5,7$. 
Hence $H\{E\nu_4,2\nu_7,\eta_{10}\circ\sigma_{11},\sigma_{18}\}_{(1,0,0,0)}^{(2)}=\bar{\sigma}_9$. 
This ends the proof of Corollary 1.2. 
\end{proof}

\section{4-fold Toda brackets $\{\vec{\bm \alpha}\}^{(2)'}$ and the generalized Hopf invariant}

\subsection{Definition of $\{\vec{\bm \alpha}\}^{(2)'}$}

Suppose that homotopy classes $\alpha_k\in[X_k,X_{k+1}]\ (1\le k\le 4)$ are given. 
We call the sequence $\vec{\bm \alpha}=(\alpha_4,\alpha_3,\alpha_2,\alpha_1)$ {\it quasi-admissible} if $\alpha_{k+1}\circ\alpha_k=0\ (1\le k\le 3)$ and there exist a representative $\vec{\bm f}$ of $\vec{\bm \alpha}$ and a sequence $\vec{\bm A}=(A_3,A_2,A_1)$ of 
null homotopies $A_k:f_{k+1}\circ f_k\simeq *\ (1\le k\le 3)$ such that $\widetilde{E}\vec{\bm A}=(\widetilde{E}A_3,\widetilde{E}A_2,\widetilde{E}A_1)$ is an admissible triple for $E\vec{\bm f}=(Ef_4,Ef_3,Ef_2,Ef_1)$ and so $E\vec{\bm \alpha}$ is admissible. 
In the last case we call $\vec{\bm A}$ a {\it quasi-admissible triple} for $\vec{\bm f}$. 
If $\vec{\bm \alpha}$ is quasi-admissible and $\vec{\bm f'}$ is any 
representative of $\vec{\bm \alpha}$, then there exists a quasi-admissible triple for $\vec{\bm f'}$. 
In fact

\begin{lemma'}
Suppose that $\vec{\bm A}$ is a quasi-admissible triple for $\vec{\bm f}$ and $H_k:f_k\simeq f_k'\ (1\le k\le 4)$. 
If we set $A_k'=A_k\bullet((-H_{k+1})\bar\circ(-H_k))$, then 
$\vec{\bm A'}=(A_3',A_2',A_1')$ is a quasi-admissible triple for $\vec{\bm f'}=(f_4',f_3',f_2',f_1')$ and 
\begin{gather}
\{E\vec{\bm f}\,;\widetilde{E}\vec{\bm A}\}^{(2)}=
\{E\vec{\bm f'};\widetilde{E}\vec{\bm A'}\}^{(2)},\\
\bigcup_{\vec{\bm A}}\{E\vec{\bm f}\,;\widetilde{E}\vec{\bm A}\}^{(2)}
=\bigcup_{\vec{\bm B'}}\{E\vec{\bm f'};\widetilde{E}\vec{\bm B'}\}^{(2)},
\end{gather}
where $\bigcup_{\vec{\bm A}}$ and $\bigcup_{\vec{\bm B'}}$ are unions taken over all quasi-admissible triples $\vec{\bm A}$ and $\vec{\bm B'}$ for $\vec{\bm f}$ and $\vec{\bm f'}$ respectively.
\end{lemma'}
\begin{proof}
The equality (3.1.1) holds by \cite[Proposition 2.11]{OO1}. 
Hence $\bigcup_{\vec{\bm A}}\{E\vec{\bm f}\,;\widetilde{E}\vec{\bm A}\}^{(2)}\subset\bigcup_{\vec{\bm B'}}\{E\vec{\bm f'};\widetilde{E}\vec{\bm B'}\}^{(2)}$.  
Interchanging $\vec{\bm f}$ and $\vec{\bm f'}$ each other we obtain the opposite containment. 
Hence (3.1.2) holds. This ends the proof.
\end{proof}

As is easily seen, if $\vec{\bm \alpha}$ is admissible, then it is quasi-admissible. 
As shall be proved in Corollary 1.4, quasi-admissible does not imply admissible in general. 
Another example is

\begin{exam'}
$(2\iota_2,\eta_2^2,2\iota_4,\eta_4)$ is not admissible but quasi-admissible. 
\end{exam'}

We omit a proof of the above example because we do not use it in this note. 

Suppose that $\vec{\bm \alpha}$ is quasi-admissible. 
For a quasi-admissible triple $\vec{\bm A}$ for a representative $\vec{\bm f}$ of $\vec{\bm \alpha}$, we defined, in \S2.1, 
\begin{equation}
\begin{split}
&\{E\vec{\bm f}\,;\widetilde{E}\vec{\bm A}\}^{(2)}
=\{Ef_4,[Ef_3,\widetilde{E}A_2,Ef_2],(Ef_2,\widetilde{E}A_1,Ef_1)\}\\
&\hspace{4cm}\cap\{[Ef_4,\widetilde{E}A_3,Ef_3],(Ef_3,\widetilde{E}A_2,Ef_2),-E^2f_1\}
\end{split}
\end{equation}
which is a non-empty subset of $[E^3X_1,EX_5]$ by  Proposition A.2. 
We define 
$\{\vec{\bm f}\}^{(2)'}=\bigcup\{E\vec{\bm f}\,;\widetilde{E}\vec{\bm A}\}^{(2)}$, where the union $\bigcup$ is taken over all quasi-admissible triples $\vec{\bm A}$ for $\vec{\bm f}$. 
Since $\{\vec{\bm f}\}^{(2)'}$ depends only on homotopy classes of $f_k$ by Lemma 3.1.1, we denote $\{\vec{\bm f}\}^{(2)'}$ by $\{\vec{\bm \alpha}\}^{(2)'}$, that is, 
$$
\{\vec{\bm \alpha}\}^{(2)'}=\bigcup_{\vec{\bm A}}\{E\vec{\bm f}\,;\widetilde{E}\vec{\bm A}\}^{(2)}.
$$
From the definitions, if $\vec{\bm \alpha}$ is quasi-admissible, then 
\begin{equation}
\{\vec{\bm \alpha}\}^{(2)'}\subset\{E\vec{\bm \alpha}\}^{(2)}. 
\end{equation}

\begin{lemma'}
If $\vec{\bm A}$ is a quasi-admissible triple for a representative $\vec{\bm f}$ of $\vec{\bm \alpha}$ and if $X_k\ (k=1,3,4)$ is a CW-complex with a vertex as the base point and
\begin{equation}
1+\dim X_1\le 2\cdot \mathrm{Min}\{\mathrm{conn}(X_4), \mathrm{conn}(X_4\cup_{f_3}CX_3)\},
\end{equation}
where $\mathrm{conn}(X)$ denotes the connectivity of the space $X$, then $(f_3,A_2,f_2)\circ Ef_1\simeq *$ and 
$$
\{E[f_4,A_3,f_3],E(f_3,A_2,f_2),E^2f_1\}=\{E[f_4,A_3,f_3],(f_3,A_2,f_2),Ef_1\}_1.
$$
\end{lemma'}
\begin{proof}
We have
\begin{align*}
E([f_3,&A_2,f_2]\circ(f_2,A_1,f_1))\\
&=[Ef_3,\widetilde{E}A_2,Ef_2]\circ(Ef_2,\widetilde{E}A_1,Ef_1)\circ(1_{X_1}\wedge\tau(\s^1,\s^1))\quad(\text{by \cite[Lemma 2.4]{OO1}})\\
&\simeq *.
\end{align*}
Since the suspension $E:[EX_1,X_4]\to [E^2X_1,EX_4]$ is an isomorphism by (3.1.5), it follows that $[f_3,A_2,f_2]\circ(f_2,A_1,f_1)\simeq *$ so that 
\begin{align*}
(f_3,&A_2,f_2)\circ Ef_1= -((f_3,A_2,f_2)\circ(-Ef_1))\simeq -((f_3,A_2,f_2)\circ q_{f_2}\circ(f_2,A_1,f_1))\\
&\simeq -(i_{f_3}\circ[f_3,A_2,f_2]\circ(f_2,A_1,f_1))
\quad(\text{by \cite[(5.11)]{Og} cf. \cite[Lemma 3.6]{OO1}})\\
&\simeq *.
\end{align*}
By (3.1.5), two sets below are the same 
and so the inclusion 
$
\{E[f_4,A_3,f_3],(f_3,A_2,f_2),Ef_1\}_1\subset\{E[f_4,A_3,f_3],E(f_3,A_2,f_2),E^2f_1\}
$ 
is the equality $=$.
\begin{align*}
\mathrm{Indet}&\{E[f_4,A_3,f_3],(f_3,A_2,f_2),Ef_1\}_1\\
&=[E^3X_2,EX_5]\circ E^3f_1+E[f_4,A_3,f_3]\circ E[E^2X_1,X_4\cup_{f_3}CX_3],\\
\mathrm{Indet}&\{E[f_4,A_3,f_3],E(f_3,A_2,f_2),E^2f_1\}\\
&=[E^3X_2,EX_5]\circ E^3 f_1+E[f_4,A_3,f_3]\circ[E^3X_1,E(X_4\cup_{f_3}CX_3)].
\end{align*}
This ends the proof.
\end{proof}

\begin{proof}[Proof of Theorem 1.3]
Let $\vec{\bm A}$ be a quasi-admissible triple for a representative $\vec{\bm f}$ of $\vec{\bm \alpha}$. 
We have 
\begin{align*}
\{E\vec{\bm f}\,;\widetilde{E}\vec{\bm A}\}^{(2)}&\subset\{[Ef_4,\widetilde{E}A_3,Ef_3],(Ef_3,\widetilde{E}A_2,Ef_2),-E^2f_1\}\quad(\text{by (3.1.3)})\\
&=\{E[f_4,A_3,f_3], E(f_3,A_2,f_2),E^2f_1\}\quad(\text{by \cite[Lemma 2.4]{OO1}})\\
&=\{E[f_4,A_3,f_3], (f_3,A_2,f_2),Ef_1\}_1\quad(\text{by Lemma 3.1.3})
\end{align*}
and so 
\begin{align*}
H\{E\vec{\bm f}\,;\widetilde{E}\vec{\bm A}\}^{(2)}
&\subset H\{E[f_4,A_3,f_3],(f_3,A_2,f_2),Ef_1\}_1\\
&=\Delta^{-1}([f_4,A_3,f_3]\circ(f_3,A_2,f_2))\circ E^3f_1\quad (\text{by \cite[Proposition 2.6]{T}})\\
&\subset \Delta^{-1}(\{f_4,f_3,f_2\})\circ E^3 f_1
=\Delta^{-1}(\{\alpha_4,\alpha_3,\alpha_2\})\circ E^3\alpha_1.
\end{align*}
Taking the union over all quasi-admissible triples $\vec{\bm A}$ for $\vec{\bm f}$, we have 
$$
H\{\vec{\bm \alpha}\}^{(2)'}\subset \Delta^{-1}(\{\alpha_4,\alpha_3,\alpha_2\})\circ E^3\alpha_1.
$$ 
This ends the proof.
\end{proof}

\subsection{Proof of Corollary 1.4}

To prove Corollary 1.4 we need three lemmas. 
Lemma~3.2.1 below overlaps \cite[p.185]{M1} and \cite[Proposition 10]{MO}. 

\begin{lemma'}
If $n\ge 12$, then 
$\{\nu_n,2\nu_{n+3},\nu_{n+6}\}=\{2\nu_n,\nu_{n+3},2\nu_{n+6}\}=\{0\}$ and so the quadruple $(\nu_n,2\nu_{n+3},\nu_{n+6},2\nu_{n+9})$ is admissible. 
\end{lemma'}

\begin{lemma'}
Let $n\ge 15$. 
Let $\vec{\bm f}=(f_4,f_3,f_2,f_1)$ be a representative of $(\nu_n,2\nu_{n+3},\nu_{n+6},2\nu_{n+9})$ and $(A_3,A_2,A_1)$ an admissible triple for $\vec{\bm f}$. 
Then 
\begin{gather}
\eta_{n-1}\circ\{[f_4,A_3,f_3],(f_3,A_2,f_2),-Ef_1\}=\eta_{n-1}\circ\kappa_n,\\
\{[f_4,A_3,f_3],(f_3,A_2,f_2),-Ef_1\}\equiv \kappa_n\ \mathrm{mod}\ \sigma_n^2,\\
\eta_{n-1}\circ\{\nu_n,2\nu_{n+3},\nu_{n+6},2\nu_{n+9}\}^{(2)}=\eta_{n-1}\circ\kappa_n,\\
\{\nu_n,2\nu_{n+3},\nu_{n+6},2\nu_{n+9}\}^{(2)}\equiv \kappa_n\ \mathrm{mod}\ \sigma_n^2.
\end{gather}
\end{lemma'}

\begin{lemma'}
$\{\sigma_{12},\nu_{19},2\nu_{22}\}=\Delta\nu_{25}+\bZ_{16}\{\sigma_{12}^2\}\oplus\bZ_2\{2\Delta\nu_{25}\}$.
\end{lemma'}

Before proving three lemmas, we prove Corollary 1.4 from three lemmas. 
Since 
\begin{equation}
\pi_{26}^{12}=\bZ_{16}\{\sigma_{12}^2\}\oplus\bZ_2\{\kappa_{12}\}\oplus\bZ_4\{\Delta\nu_{25}\}
\end{equation}
by \cite[Theorem 10.3]{T}, it follows from Lemma 3.2.3 that 
$\{\sigma_{12},\nu_{19},2\nu_{22}\}$ does not contain $0$ so that 
$(\sigma_{12},\nu_{19},2\nu_{22},\nu_{25})$ is not admissible. 
To economize on notations, we use the same letter for a map and its homotopy class. 
By Lemma 3.2.3 we can take $A_3:\sigma_{12}\circ\nu_{19}\simeq *$ and $A_2:\nu_{19}\circ 2\nu_{22}\simeq *$ such that $\Delta\nu_{25}=[\sigma_{12},A_3,\nu_{19}]\circ(\nu_{19},A_2,2\nu_{22})$. 
Take $A_1:2\nu_{22}\circ\nu_{25}\simeq *$ arbitrarily. 
Then $[\nu_{19},A_2,2\nu_{22}]\circ(2\nu_{22},A_1,\nu_{25})\simeq *$ by Lemma 3.2.1. 
We have
\begin{align*}
0&=E\Delta\nu_{25}=E([\sigma_{12},A_3,\nu_{19}]\circ(\nu_{19},A_2,2\nu_{22}))\\
&=[E\sigma_{12},\widetilde{E}A_3,E\nu_{19}]\circ(E\nu_{19},\widetilde{E}A_2,E(2\nu_{22}))\circ(1_{\s^{25}}\wedge\tau(\s^1,\s^1))
\quad(\text{by \cite[Lemma 2.4]{OO1}})
\end{align*}
so that $[E\sigma_{12},\widetilde{E}A_3,E\nu_{19}]\circ(E\nu_{19},\widetilde{E}A_2,E(2\nu_{22}))\simeq *$. 
Similarly we have $[E\nu_{19},\widetilde{E}A_2,E(2\nu_{22})]\circ(E(2\nu_{22}),\widetilde{E}A_1,E\nu_{25})\simeq *$. 
Therefore $(\sigma_{12},\nu_{19},2\nu_{22},\nu_{25})$ is 
quasi-admissible. 
This proves Corollary 1.4(1). 

By Theorem 1.3 and Lemma 3.2.3, we have
$$
H\{\sigma_{12},\nu_{19},2\nu_{22},\nu_{25}\}^{(2)'}
\subset \Delta^{-1}(\{\sigma_{12},\nu_{19},2\nu_{22}\})\circ\nu_{28}
=\{\nu_{25}^2\}.
$$
This proves Corollary 1.4(2). 

Let $\lambda_0\in\{\sigma_{12},\nu_{19},2\nu_{22},\nu_{25}\}^{(2)'}$. 
Then $\lambda_0\in\{\sigma_{13},\nu_{20},2\nu_{23},\nu_{26}\}^{(2)}$ 
by (3.1.4). 
Hence, by definitions, there exists an admissible triple $(B_3,B_2,B_1)$ for $(\sigma_{13},\nu_{20},2\nu_{23},\nu_{26})$ such that $\lambda_0\in\{\sigma_{13},[\nu_{20},B_2,2\nu_{23}],(2\nu_{23},B_1,\nu_{26})\}$. 
Take $B_0:\nu_{26}\circ 2\nu_{29}\simeq *$ arbitrarily. 
Then $(B_2,B_1,B_0)$ is an admissible triple for $(\nu_{20},2\nu_{23},\nu_{26},2\nu_{29})$ by Lemma 3.2.1. 
We have
\begin{align*}
(2\nu_{23},B_1,\nu_{26})\circ(-2\nu_{30})&\simeq(2\nu_{23},B_1,\nu_{26})\circ q_{\nu_{26}}\circ(\nu_{26},B_0,2\nu_{29})\\
&\simeq i_{2\nu_{23}}\circ[2\nu_{23},B_1,\nu_{26}]\circ(\nu_{26},B_0,2\nu_{29})\simeq *
\end{align*}
and so $(2\nu_{23},B_1,\nu_{26})\circ 2\nu_{30}\simeq *$. 
We then have
\begin{align*}
\lambda_0\circ 2\nu_{31}\in&\{\sigma_{13},[\nu_{20},B_2,2\nu_{23}],(2\nu_{23},B_1,\nu_{26})\}\circ 2\nu_{31}\\
&=-(\sigma_{13}\circ\{[\nu_{20},B_2,2\nu_{23}],(2\nu_{23},B_1,\nu_{26}), 2\nu_{30}\})\quad(\text{by \cite[Proposition 1.4]{T}})\\
&=\sigma_{13}\circ\{[\nu_{20},B_2,2\nu_{23}],(2\nu_{23},B_1,\nu_{26}), -2\nu_{30}\}.
\end{align*}
By (3.2.2), we have
$$
\{[\nu_{20},B_2,2\nu_{23}],(2\nu_{23},B_1,\nu_{26}), -2\nu_{30}\}\equiv 
\kappa_{20}\ \mathrm{mod}\ \sigma_{20}^2.
$$
Hence $\lambda_0\circ 2\nu_{31}\equiv \sigma_{13}\circ\kappa_{20}\ \mathrm{mod}\ \sigma_{13}^3$. 
As seen in \cite[p.308]{M2}, $\pi_{34}^{13}$ contains $\bZ_2\{\sigma_{13}\circ\kappa_{20}\}\oplus\bZ_2\{\sigma^3_{13}\}$. 
Hence the order of $\lambda_0\circ\nu_{31}$ is four. 
This proves Corollary 1.4(3). 

In the rest of the subsection we will prove three lemmas. 

\begin{proof}[Proof of Lemma 3.2.1]
It follows from \cite[(7.20)]{T} that 
$$
\mathrm{Indet}\{\nu_n,2\nu_{n+3},\nu_{n+6}\}=\mathrm{Indet}\{2\nu_n,\nu_{n+3},2\nu_{n+6}\}=\mathrm{Indet}\{2\nu_n,2\nu_{n+3},\nu_{n+6}\}=\{0\}
$$
so that each of three brackets consists of a single element. 
Set $x=\{\nu_n,2\nu_{n+3},\nu_{n+6}\}\in\pi_{n+10}^n=\bZ_2\{\eta_n\circ\mu_{n+1}\}$. 
Of course $x$ is $0$ or $\eta_n\circ\mu_{n+1}$. 
We will show $x=0$. 
We have
\begin{align*}
x\circ\eta_{n+10}&=-(\nu_n\circ\{2\nu_{n+3},\nu_{n+6},\eta_{n+9}\})\quad(\text{by \cite[Proposition 1.4]{T}})\\
&\in -(\nu_n\circ\pi_{n+11}^{n+3})=-(\nu_n\circ\{\bar{\nu}_{n+3},\varepsilon_{n+3}\})\\
&=\nu_n\circ\varepsilon_{n+3}\quad(\text{by \cite[(7.17)]{T}})\\
&=E^{n-6}\Delta(\nu_{13}^2)\quad(\text{by \cite[(7.18)]{T}})\\
&=0,
\end{align*}
that is, $x\circ\eta_{n+10}=0$. 
If $x=\eta_n\circ\mu_{n+1}$, then 
\begin{align*}
x\circ\eta_{n+10}&=\eta_n\circ\mu_{n+1}\circ\eta_{n+10}\\
&=\eta_n^2\circ\mu_{n+2}\quad(\text{by \cite[Proposition 3.1]{T}})\\
&=4\zeta_n\quad(\text{by \cite[(7.14)]{T}})\\
&\ne 0\quad(\text{by \cite[Theorem 7.4]{T}}).
\end{align*}
This is a contradiction. 
Hence $x=0$, that is, $\{\nu_n,2\nu_{n+3},\nu_{n+6}\}=\{0\}$. 

We have $\varepsilon'=\{\nu',\nu',\nu_6\}_3\in\pi_{13}^3$ by \cite[p.58]{T} and  
\begin{align*}
E^{n-3}\varepsilon'&=\pm 2\nu_n\circ\sigma_{n+3}\quad(\text{by \cite[(7.10)]{T}})\\
&=0\quad (\text{by \cite[(7.20)]{T}}).
\end{align*}
Hence we have $0=E^{n-3}\varepsilon'\in\{E^{n-3}\nu',E^{n-3}\nu',E^{n-3}\nu_6\}_3\subset\{2\nu_n,2\nu_{n+3},\nu_{n+6}\}$ and so 
$\{2\nu_n,2\nu_{n+3},\nu_{n+6}\}=\{0\}$. 
We have $\{2\nu_n,\nu_{n+3},2\nu_{n+6}\}=\{0\}$, since $\{2\nu_n,2\nu_{n+3},\nu_{n+6}\}\supset\{2\nu_n,\nu_{n+3},2\nu_{n+6}\}$. 
Hence $(\nu_n,2\nu_{n+3},\nu_{n+6},2\nu_{n+9})$ is admissible. 
This completes the proof of Lemma 3.2.1. 
\end{proof}

\begin{proof}[Proof of Lemma 3.2.2]
Let $n\ge 15$. 
(3.2.3) and (3.2.4) follow from (3.2.1) and (3.2.2) respectively. 

As before we use the same letter for a map and its homotopy class. 
To prove (3.2.1), note that $
(2\nu_{n+3},A_2,\nu_{n+6})\circ(-E(2\nu_{n+9}))\simeq i_{2\nu_{n+3}}\circ[2\nu_{n+3},A_2,\nu_{n+6}]\circ(\nu_{n+6},A_1,2\nu_{n+9})\simeq *
$ and  
\begin{align*}
\eta_{n-1}\circ\{[\nu_n,A_3,2\nu_{n+3}],&(2\nu_{n+3},A_2,\nu_{n+6}),-E(2\nu_{n+9})\}\\
&\subset\{\eta_{n-1}\circ [\nu_n,A_3,2\nu_{n+3}],(2\nu_{n+3},A_2,\nu_{n+6}),-E(2\nu_{n+9})\},
\end{align*}
where $\eta_{n-1}\circ[\nu_n,A_3,2\nu_{n+3}]\in\{\eta_{n-1},\nu_n,2\nu_{n+3}\}\circ q_{2\nu_{n+3}}$ by \cite[Proposition 1.9]{T}. 
We have 
$$
\mathrm{Indet}\{\eta_{n-1},\nu_n,2\nu_{n+3}\}=\mathrm{Indet}\{\eta_{n-1},2\nu_n,\nu_{n+3}\}
=\eta_{n-1}\circ\pi_{n+7}^n=\bZ_2\{\eta_{n-1}\circ\sigma_n\}.
$$
Since $\{\eta_{n-1},2\nu_n,\nu_{n+3}\}\ni\varepsilon_{n-1}$ by \cite[(6.1)]{T} and 
$\{\eta_{n-1},\nu_n,2\nu_{n+3}\}\subset\{\eta_{n-1},2\nu_n,\nu_{n+3}\}$, we have 
$\{\eta_{n-1},\nu_n,2\nu_{n+3}\}=\{\eta_{n-1},2\nu_n,\nu_{n+3}\}=\{\varepsilon_{n-1},\varepsilon_{n-1}+\eta_{n-1}\circ\sigma_n\}$. 
Hence $\eta_{n-1}\circ[\nu_n,A_3,2\nu_{n+3}]$ is $\varepsilon_{n-1}\circ q_{2\nu_{n+3}}$ or 
$(\varepsilon_{n-1}+\eta_{n-1}\circ\sigma_n)\circ q_{2\nu_{n+3}}$. 
We have 
\begin{align*}
\{\varepsilon_{n-1}&\circ q_{2\nu_{n+3}},(2\nu_{n+3},A_2,\nu_{n+6}),-E(2\nu_{n+9})\}\\
&\subset\{\varepsilon_{n-1},-E\nu_{n+6},-E(2\nu_{n+9})\}=\{\varepsilon_{n-1},\nu_{n+7},2\nu_{n+10}\}\\
&\subset\{\varepsilon_{n-1},2\nu_{n+7},\nu_{n+10}\}\\
&\supset\{\varepsilon_{n-1},2\iota_{n+7},\nu_{n+7}^2\}
\supset(-1)^{n-4}E^{n-4}\{\varepsilon_3,2\iota_{11},\nu_{11}^2\}
\supset(-1)^nE^{n-4}\{\varepsilon_3,2\iota_5,\nu_5^2\}_6\\
&\ni E^{n-4}\bar{\varepsilon}_3=\bar{\varepsilon}_{n-1}=\eta_{n-1}\circ\kappa_n\quad (\text{by \cite[(10.23)]{T}}).
\end{align*}
Since $\pi_{n+11}(\s^{n-1})=0$, by indeterminacies, we have 
$$
\{\varepsilon_{n-1},\nu_{n+7},2\nu_{n+10}\}=\{\varepsilon_{n-1},2\nu_{n+7},\nu_{n+10}\}.
$$
Since $\mathrm{Indet}\{\varepsilon_{n-1},2\nu_{n+7},\nu_{n+10}\}=\varepsilon_{n-1}\circ\pi_{n+14}^{n+7}=0$ by \cite[Lemma 10.7]{T}, 
we have
$$
\{\varepsilon_{n-1},\nu_{n+7},2\nu_{n+10}\}=\{\varepsilon_{n-1},2\nu_{n+7},\nu_{n+10}\}=\bar{\varepsilon}_{n-1}=\eta_{n-1}\circ\kappa_n
$$
so that $\{\varepsilon_{n-1}\circ q_{2\nu_{n+3}},(2\nu_{n+3},A_2,\nu_{n+6}),-E(2\nu_{n+9})\}=\eta_{n-1}\circ\kappa_n$. 
On the other hand, we have
\begin{align*}
&\{(\varepsilon_{n-1}+\eta_{n-1}\circ\sigma_n)\circ q_{2\nu_{n+3}},(2\nu_{n+3},A_2,\nu_{n+6}),-E(2\nu_{n+9})\}\\
&\subset\{\varepsilon_{n-1}+\eta_{n-1}\circ\sigma_n,-E(\nu_{n+6}),-E(2\nu_{n+9})\}=
\{\varepsilon_{n-1}+\eta_{n-1}\circ\sigma_n,E\nu_{n+6},E(2\nu_{n+9})\}\\
&\subset\{\varepsilon_{n-1},\nu_{n+7},2\nu_{n+10}\}+\{\eta_{n-1}\circ\sigma_n,\nu_{n+7},2\nu_{n+10}\}.
\end{align*}
To complete the proof of (3.2.1), it suffices to prove $\{\eta_{n-1}\circ\sigma_n,\nu_{n+7},2\nu_{n+10}\}=\{0\}$. 
Since $\eta_{n-1}\circ\sigma_n=\sigma_{n-1}\circ\eta_{n+6}$ by \cite[Proposition 3.1]{T}, we have
\begin{align*}
\{\eta_{n-1}\circ\sigma_n,\nu_{n+7},2\nu_{n+10}\}&=\{\sigma_{n-1}\circ\eta_{n+6},\nu_{n+7},2\nu_{n+10}\}\\
&\supset\sigma_{n-1}\circ\{\eta_{n+6},\nu_{n+7},2\nu_{n+10}\}\\
&\ni \sigma_{n-1}\circ\varepsilon_{n+6}\quad(\text{by \cite[(6.1)]{T}})\\
&=0\quad (\text{by \cite[Lemma 10.7]{T}}).
\end{align*}
Hence 
\begin{align*}
\{\eta_{n-1}\circ\sigma_n,\nu_{n+7},2\nu_{n+10}\}
&=\mathrm{Indet}\{\eta_{n-1}\circ\sigma_n,\nu_{n+7},2\nu_{n+10}\}\\
&=\pi_{n+11}(\s^{n-1})\circ 2\nu_{n+11}+\eta_{n-1}\circ\sigma_n\circ\pi_{n+14}(\s^{n+7})\\
&=0\quad (\text{by \cite[Theorem 7.6, Lemma 10.7]{T}}).
\end{align*}
That is, $\{\eta_{n-1}\circ\sigma_n,\nu_{n+7},2\nu_{n+10}\}=\{0\}$ as desired. 

Next we prove (3.2.2). 
By Lemma 6.4, Theorem 10.3, Lemma 10.7, and Theorem 10.10 of \cite{T}, we have
\begin{gather*}
\pi_{n+14}^n=\begin{cases} \bZ_4\{\sigma_{15}^2\}\oplus\bZ_2\{\kappa_{15}\} & n=15\\
(\bZ_2)^2\{\sigma_n^2, \kappa_n\} & n\ge 16 \end{cases},\\
\pi_{n+14}^{n-1}=\begin{cases} \bZ_{16}\{\rho_{n-1}\}\oplus\bZ_2\{\bar{\varepsilon}_{n-1}\} & n=15, 16, n\ge 18\\
\bZ_{32}\{\rho_{16}\}\oplus\bZ_2\{\bar{\varepsilon}_{16}\}\oplus\bZ\{\Delta\iota_{33}\} & n=17\end{cases},\\
\eta_{n-1}\circ\sigma_n^2=0\ (n\ge 10),\quad \eta_{n-1}\circ\kappa_n=\bar{\varepsilon}_{n-1}\ (n\ge 7).
\end{gather*}
Then by (3.2.1) we have $\{[\nu_n,A_3,2\nu_{n+3}],(2\nu_{n+3},A_2,\nu_{n+6}),-E(2\nu_{n+9})\}\equiv \kappa_n\ \mathrm{mod}\ \sigma_n^2$. 
This proves (3.2.2) and completes the proof of Lemma 3.2.2.
\end{proof}

\begin{proof}[Proof of Lemma 3.2.3]
We have 
\begin{align*}
\mathrm{Indet}\{\sigma_{12},\nu_{18},2\nu_{21}\}_1&=\pi^{12}_{23}\circ 2\nu_{23}+\sigma_{12}\circ E\pi^{18}_{25}\\
&=\pi^{12}_{23}\circ 2\nu_{23}+\sigma_{12}\circ \pi^{19}_{26}=\mathrm{Indet}\{\sigma_{12},\nu_{19},2\nu_{22}\}\\
&=\big(\bZ\{\Delta\iota_{25}\}\oplus\bZ_8\{\zeta_{12}\}\big)\circ 2\nu_{23}+\sigma_{12}\circ\bZ_{16}\{\sigma_{19}\}\\
&\hspace{2cm}(\text{by \cite[Theorem 7.4, Proposition 5.15]{T}}).
\end{align*}
Since $\Delta\iota_{25}\circ\nu_{23}=\Delta\nu_{23}$ 
by \cite[Proposition 2.5]{T}, we have $\bZ\{\Delta\iota_{25}\}\circ 2\nu_{23}=\bZ_2\{2\Delta\nu_{25}\}$ by (3.2.5). 
We have 
$\zeta_9\circ\nu_{20}=-\nu_9\circ\zeta_{12}$ by \cite[Theorem 2.4]{BH} (cf.\cite[Proposition 3.1]{T}) 
so that 
$\zeta_{12}\circ 2\nu_{23}=-2\nu_{12}\circ\zeta_{15}=0$ since $2\nu_5\circ\zeta_8=0$ by \cite[Theorem 10.3]{T}. 
Therefore 
$\mathrm{Indet}\{\sigma_{12},\nu_{18},2\nu_{21}\}_1=\bZ_2\{2\Delta\nu_{25}\}\oplus\bZ_{16}\{\sigma_{12}^2\}$ by (3.10). 
Hence we can take an element $\alpha\in\{\sigma_{12},\nu_{18},2\nu_{21}\}_1$ such that $\alpha=x\kappa_{12}+y\Delta\nu_{25}\ (x,y\in\{0,1\})$. 
On the other hand
\begin{align*}
H\{\sigma_{12},&\nu_{18},2\nu_{21}\}_1\\
&=\Delta^{-1}(\sigma_{11}\circ\nu_{18})\circ 2\nu_{23}\quad(\text{by \cite[Proposition 2.6]{T} cf. near the end of \S3.1})\\
&=\{2\nu_{23},-2\nu_{23}\}\quad(\text{since $E(\sigma_{11}\circ\nu_{18})=0$ by \cite[(7.20)]{T} and so $\Delta\iota_{23}=\sigma_{11}\circ\nu_{18}$}),\\
H(\alpha)&=y H(\Delta\nu_{25})\\
&=y H(\pm[\iota_{12},\iota_{12}]\circ\nu_{23})\quad(\text{by \cite[Proposition 2.5]{T}})\\
&=\pm 2y\nu_{23} \quad(\text{by \cite[Proposition 2.2]{T}}).
\end{align*}
Hence $y=1$ 
and so we have
$$
\{\sigma_{12},\nu_{19},2\nu_{22}\}\circ\eta_{26}\ni\alpha\circ\eta_{26}=x\kappa_{12}\circ\eta_{26}+(\Delta\nu_{25})\circ\eta_{26}=x\kappa_{12}\circ\eta_{26}
$$
by \cite[Proposition 2.5]{T}. 
On the other hand 
$$
\{\sigma_{12},\nu_{19},2\nu_{22}\}\circ\eta_{26}=-(\sigma_{12}\circ\{\nu_{19},2\nu_{22},\eta_{25}\})\subset\sigma_{12}\circ\pi_{27}^{19}=\{0\}
$$
by \cite[Proposition 1.4, Lemma 10.7, (10.18)]{T}. 
Hence $x\kappa_{12}\circ\eta_{26}=0$. 
Since $\kappa_{12}\circ\eta_{26}=\bar{\varepsilon}_{12}\ne 0$ by \cite[Theorem 10.5, (10.23)]{T}, we have  $x=0$. 
This completes the proof of Lemma 3.2.3.
\end{proof}

\begin{appendix}
\section{An easy proof of \cite[Theorem 4.9]{OO1}}
As an introduction to the definition of \^{O}guchi's $4$-fold bracket, we proved Theorem 4.9 of \cite{OO1} by a slightly difficult method. 
We should note that the main part of the theorem itself is an easy consequence of elementary properties of Toda brackets. 
We will explain it. 
First we prove 

\begin{lemma}
Let the following homotopy commutative diagrams be given.  
$$
\begin{CD}
X_4@<f_3<<E^nX_3@. ,\quad X_3@<f_2<<X_2@<f_1<<X_1@.,\quad f_3\circ  E^nf_2\simeq *,\ f_2\circ f_1\simeq *\\
@Va_4VV @V E^na_3VV \quad @VVa_3V @VVa_2V @VVa_1V @.@.\\
X_4'@<<f_3'<  E^nX_3'@. ,\quad X_3'@<<f_2'< X_2'@<<f_1'<X_1'@.,\quad f_3'\circ  E^nf_2'\simeq *,\ f_2'\circ f_1'\simeq *
\end{CD}
$$
\begin{enumerate}
\item([cf.Theorem (4.3) of \cite{S}])  We have $a_4\circ\{f_3,f_2,f_1\}_n=\{f_3',f_2',f_1'\}_n\circ E^{n+1}a_1$ in 
$[E^{n+1} X_2,X_4']\circ E^{n+1} f_1\backslash[E^{n+1} X_1,X_4']/f_3'\circ E^n[E X_1, X_3']$. 
\item If all $a_i$ are homotopy equivalences, then 
$a_4\circ\{f_3,f_2,f_1\}_n=\{f_3',f_2',f_1'\}_n\circ E^{n+1}a_1$ in $[E^{n+1}X_1,X_4']$. 
\item If $X_k=X_k'$ and $a_k=1_{X_k}$ for $k=1,4$, then $\{f_3,f_2,f_1\}_n$ and $\{f_3',f_2',f_1'\}_n$ have a common element. 
\end{enumerate}
\end{lemma}
\begin{proof}
(1) It follows from \cite[Proposition 1.2]{T} that we have
\begin{align*}
a_4\circ\{f_3,f_2,f_1\}_n&\subset\{a_4\circ f_3,f_2,f_1\}_n=\{f_3'\circ  E^na_3,f_2,f_1\}_n\\
&\subset\{f_3',a_3\circ f_2,f_1\}_n=\{f_3',f_2'\circ a_2,f_1\}_n\\
&\supset\{f_3',f_2',a_2\circ f_1\}_n=\{f_3',f_2',f_1'\circ a_1\}_n\\
&\supset\{f_3',f_2',f_1'\}_n\circ  E^{n+1}a_1.
\end{align*}
Hence $a_4\circ\{f_3,f_2,f_1\}_n\subset\{f_3',a_3\circ f_2,f_1\}_n
\supset\{f_3',f_2',f_1'\}_n\circ  E^{n+1}a_1$. 
This proves (1). 

(2) Suppose that $a_i\ (1\le i\le 4)$ are homotopy equivalences. 
We are going to show that four containments in the proof of (1) are equalities $=$. 
We have $a_4\circ\{f_3,f_2,f_1\}_n\subset \{a_4\circ f_3,f_2,f_1\}_n$ and so $\{\vec{\bm f}\}_n\subset a_4^{-1}\circ\{a_4\circ f_3,f_2,f_1\}_n\subset\{\vec{\bm f}\}_n$. 
Hence $\{\vec{\bm f}\}_n=a_4^{-1}\circ\{a_4\circ f_3,f_2,f_1\}_n$ and 
$a_4\circ\{\vec{\bm f}\}_n=\{a_4\circ f_3,f_2,f_1\}_n$. 
We have $\{f'_3,a_3\circ f_2,f_1\}_n\supset\{f'_3\circ E^n a_3, f_2,f_1\}_n=
\{f'_3\circ E^na_3,a_3^{-1}\circ a_3\circ f_2,f_1\}_n\supset\{f'_3,a_3\circ f_2,f_1\}_n$ and so $\{f'_3\circ E^n a_3, f_2, f_1\}_n=\{f'_3,a_3\circ f_2,f_1\}_n$. 
We have $\{f'_3,f'_2\circ a_2,f_1\}_n\supset\{f'_3,f'_2,a_2\circ f_1\}_n=\{f'_3,f'_2\circ a_2\circ a_2^{-1},a_2\circ f_1\}_n\supset\{f'_3,f'_2\circ a_2,f_1\}_n$ and so $\{f'_3,f'_2\circ a_2,f_1\}_n=\{f'_3,f'_2,a_2\circ f_1\}_n$. 
We have $\{f'_3,f'_2,f'_1\circ a_1\}_n\supset\{f'_3,f'_2,f'_1\}_n\circ E^{n+1}a_1=\{f'_3,f'_2,f'_1\circ a_1\circ a_1^{-1}\}_n\circ E^{n+1}a_1\supset 
\{f'_3,f'_2,f'_1\circ a_1\}_n$ and so $\{f'_3,f'_2,f'_1\circ a_1\}_n=\{\vec{\bm f'}\}_n\circ E^{n+1}a_1$. 
This proves (2). 

(3) Suppose that $X_i=X_i'$ and $a_i=1_{X_i}$ for $i=1,4$. 
Take $\alpha\in\{f_3,f_2,f_1\}_n$ and $\alpha'\in\{f_3',f_2',f_1'\}_n$ arbitrarily. 
Then we have 
\begin{gather*}
\{f_3,f_2,f_1\}_n=[E^{n+1}X_2,X_4]\circ E^{n+1}f_1+\alpha+f_3\circ E^n[EX_1,X_3],\\
\{f_3',f_2',f_1'\}_n=[E^{n+1}X_2',X_4]\circ E^{n+1}f_1'+\alpha'+f_3'\circ E^n[EX_1,X_3'].
\end{gather*}
Since $a_2\circ f_1\simeq f_1'$ and $f_3\simeq f_3'\circ E^n a_3$, we have $[E^{n+1}X_2,X_4]\circ E^{n+1}f_1\supset[E^{n+1}X_2',X_4]\circ E^{n+1}f_1'$ and $f_3\circ E^n[EX_1,X_3]\subset f_3'\circ E^n[EX_1,X_3']$. 
It follows from (1) that we have  
\begin{align*}
&[E^{n+1}X_2,X_4]\circ E^{n+1}f_1+\alpha+f_3'\circ E^n[EX_1,X_3']\\
&=[E^{n+1}X_2,X_4]\circ E^{n+1}f_1+\alpha'+f_3'\circ E^n[EX_1,X_3']
\end{align*}
so that $\alpha'=\beta+\alpha+\gamma$ for some $\beta\in [E^{n+1}X_2,X_4]\circ E^{n+1}f_1$ and $\gamma\in f_3'\circ E^n[EX_1,X_3']$. 
Then $\beta+\alpha$ is contained in both $\{f_3,f_2,f_1\}_n$ and $\{f_3',f_2',f_1'\}_n$. 
This proves (3). 
\end{proof}

Now we rewrite the main part of Theorem 4.9 of \cite{OO1}. 
Suppose that the following data are given: five spaces $X_k\ (1\le k\le 5)$, four non-negative integers $m_k\ (1\le k\le 4)$ (usually we take $m_1=m_2=0$ because it does not injure the generality as shall be shown on the end of this appendix), four maps  $f_k:E^{m_k}X_k\to X_{k+1}\ (1\le k\le 4)$, and three null homotopies $A_k:f_{k+1}\circ E^{m_{k+1}}f_k\simeq *\ (1\le k\le 3)$ such that 
$$
[f_{k+2},A_{k+1},E^{m_{k+2}}f_{k+1}]\circ(E^{m_{k+2}}f_{k+1},\widetilde{E}^{m_{k+2}}A_k,E^{m_{k+2}}E^{m_{k+1}}f_k)\simeq *\ (k=1,2).
$$
Then the main part of Theorem 4.9 of \cite{OO1} is

\begin{prop}
The Toda brackets $\{f_4,[f_3,A_2,E^{m_3}f_2],(E^{m_3}f_2,\widetilde{E}^{m_3}A_1,E^{m_{[3,2]}}f_1)\}_{m_4}$ and $\{[f_4,A_3,E^{m_4}f_3]\circ(\psi^{m_4}_{f_3})^{-1},(f_3,A_2,E^{m_3}f_2),-EE^{m_{[3,2]}}f_1\}_{m_4}$ 
have a common element. 
\end{prop} 
\begin{proof}
Consider the following homotopy commutative diagrams.  
$$
\begin{CD}
X_5 @<f_4<<E^{m_4}X_4\\
@| @VV E^{m_4}i_{f_3}V\\
X_5 @<[f_4,A_3,E^{m_4}f_3]\circ(\psi^{m_4}_{f_3})^{-1}<<E^{m_4}(X_4\cup_{f_3}CE^{m_3}X_3)
\end{CD}
$$
$$
\tiny{
\begin{CD}
X_4@<[f_3,A_2,E^{m_3}f_2]<<E^{m_3}X_3\cup_{E^{m_3}f_2}CE^{m_{[3,2]}}X_2@<(E^{m_3}f_2,\widetilde{E}^{m_3}A_1,E^{m_{[3,2]}}f_1)<<EE^{m_{[3,1]}}X_1\\
@VV i_{f_3}V @VV q_{E^{m_3}f_2}V @|\\
X_4\cup_{f_3}CE^{m_3}X_3 @<(f_3,A_2,E^{m_3}f_2)<<EE^{m_{[3,2]}}X_2
@<-EE^{m_{[3,2]}}f_1<<EE^{m_{[3,1]}}X_1
\end{CD}
}
$$
Since 
$$
(f_3,A_2,E^{m_3}f_2)\circ(-EE^{m_{[3,2]}}f_1)
\simeq i_{f_3}\circ[f_3,A_2,E^{m_3}f_2]\circ(E^{m_3}f_2,\widetilde{E}^{m_3}A_1,E^{m_{[3,2]}}f_1)\simeq *,
$$
we can apply Lemma A.1(3) to the above diagrams and we obtain the assertion. 
\end{proof}

Given $\vec{\bm \alpha}=(\alpha_4,\alpha_3,\alpha_2,\alpha_1)$ with $\alpha_k\in[E^{m_k}X_k,X_{k+1}]\ (1\le k\le 4)$, let $\vec{\bm f}=(f_4,f_3,f_2,f_1)$ be a representative of $\vec{\bm \alpha}$. 
Set 
$$
X^*_k=\begin{cases} E^{m_{[2,1]}}X_1 & k=1\\
E^{m_2}X_2 & k=2\\
X_k & 3\le k\le 5\end{cases},\quad
m^*_k=\begin{cases} 0 & k=1,2\\
m_k & k=3,4\end{cases},\quad f^*_k=\begin{cases} E^{m_2}f_1 & k=1\\
f_k & 2\le k\le 4\end{cases}.
$$
Suppose $\vec{\bm \alpha}$ is admissible. 
Let $\vec{\bm A}=(A_3,A_2,A_1)$ be an admissible triple for $\vec{\bm f}$. 
Then $\vec{\bm A}$ is also an admissible triple for $\vec{\bm f^*}$. 
It follows that $\{\vec{\bm f}\,\}^{(2)}_{\vec{\bm m}}=\{\vec{\bm f^*}\}^{(2)}_{\vec{\bm m^*}}$. 
Hence we can denote $\{\vec{\bm f}\,\}^{(2)}_{\vec{\bm m}}$ by 
$\{\vec{\bm \alpha}\}^{(2)}_{\vec{\bm m}}$ and $\{\alpha_4,\alpha_3,\alpha_2,E^{m_2}\alpha_1\}^{(2)}_{(m_4,m_3)}$. 

\end{appendix}


\begin{thebibliography}{MMO}
\bibitem[{\bf BH}]{BH} M. G. Barratt and P. J. Hilton, On join operations in homotopy groups, Proc. London Math. Soc., {\bf 3} (1953), 430--445.
\bibitem[{\bf J1}]{J1} I. M. James, Reduced product spaces, Ann. of Math., {\bf 62} (1955), 170--197.
\bibitem[{\bf J2}]{J2} I. M. James, On the suspension triad, Ann. of Math., {\bf 63} (1956), 191--247.
\bibitem[{\bf J3}]{J3} I. M. James, The suspension triad of a sphere, Ann. of Math., {\bf 63} (1956), 407--429.
\bibitem[{\bf MO}]{MO} H. J. Marcum and N. Oda, Some classical and matrix Toda brackets in the 13- and 14-stems, Kyushu J. Math., {\bf 55} (2001), 405--428.
\bibitem[{\bf M}]{M} J. Milnor, On spaces having the homotopy type of a CW-complex, Trans. Amer. Math. Soc., {\bf 90} (1959), 272--280.
\bibitem[{\bf M1}]{M1} M. Mimura, On the generalized Hopf homomorphism and the higher composition, Part I, J. Math. Kyoto Univ., {\bf 4}-1 (1964), 171--190.
\bibitem[{\bf M2}]{M2} M. Mimura, On the generalized Hopf homomorphism and the higher composition. Part II. $\pi_{n+i}(\s^n)$ for $i=21$ and $22$, J. Math. Kyoto Univ., {\bf 4}-2 (1965), 301--326.
\bibitem[{\bf MMO}]{MMO} M. Mimura, M. Mori and N. Oda, Determination of $2$-components of the $23$ and $24$-stems in homotopy groups of spheres, Mem. Fac. Sci. Kyushu Univ. Series A, Math., {\bf 29} (1975), 1--42.
\bibitem[{\bf MT}]{MT} M. Mimura and H. Toda, The $(n+20)$-th homotopy groups of $n$-spheres, J. Math. Kyoto Univ., {\bf 3}-1 (1963), 37--58. 
\bibitem[{\bf Og}]{Og} K. \^{O}guchi, A generalization of secondary composition and applications, J. Fac. Sci. Univ. Tokyo, {\bf 10} (1963), 29--79. 
\bibitem[{\bf Os}]{Os} H. \={O}shima, Memorandums on generators of the $2$-primary component of $\pi_{29}(\s^9)$, note 2021.
\bibitem[{\bf OO1}]{OO1} H. \={O}shima and K. \={O}shima, Quasi tertiary compositions and a Toda bracket in homotopy groups of $\mathrm{SU}(3)$, Math. J. Okayama Univ., {\bf 57} (2015), 13--78. 
(Typos are pointed out in \cite[Appendix C]{OO4}.)
\bibitem[{\bf OO2}]{OO2} H. \={O}shima and K. \={O}shima, 
Unstable higher Toda brackets, Math. J. Okayama Univ., {\bf 62} (2020), 27--86. (Errors are corrected in \cite[Appendix C]{OO4}.)
\bibitem[{\bf OO3}]{OO3} H. \={O}shima and K. \={O}shima, 
Unstable higher Toda brackets II, preprint 2020, arXiv:1911.03610v2.
\bibitem[{\bf OO4}]{OO4} H. \={O}shima and K. \={O}shima, A system of unstable higher Toda brackets, 
J. Math. Soc. Japan, to appear, arXiv:2004.11047v3.
\bibitem[{\bf S}]{S} E. Spanier, Secondary operations on mappings and cohomology, Ann. of Math., {\bf 75} (1962), 260--282.
\bibitem[{\bf T}]{T} H. Toda, Composition methods in the homotopy groups of spheres, Ann. of Math. Stud., {\bf 49}, Princeton Univ., Princeton, 1962. 
\end{thebibliography}
\end{document}